\documentclass[a4paper, 10pt]{amsart}
\usepackage{amsmath}
\usepackage{amsthm}
\usepackage{amssymb}
\usepackage{framed}
\usepackage[hidelinks]{hyperref}
\usepackage{xcolor}

\DeclareMathOperator{\tr}{trace}
\DeclareMathOperator{\di}{div}
\DeclareMathOperator{\gr}{grad}
\DeclareMathOperator{\Ker}{Ker}
\DeclareMathOperator{\Scal}{Scal}
\DeclareMathOperator{\Ric}{Ric}
\DeclareMathOperator{\Hess}{Hess}
\DeclareMathOperator{\cu}{curl}
\DeclareMathOperator{\Vol}{Vol}

\DeclareMathOperator{\SD}{SDiff}

\newcommand{\norm}[1]{\Vert #1\Vert}
\newcommand{\abs}[1]{\vert #1\vert}
\newcommand{\ov}{\overline}
\newcommand{\dif}{\mathrm{d}}
\newcommand{\ii}{\mathrm{i}}
\newcommand{\CC}{\mathbb{C}}

\newcommand{\DD}{\mathcal{D}}
\newcommand{\Ff}{\mathcal{F}}

\newcommand{\Oo}{\mathcal{O}}
\newcommand{\RR}{\mathbb{R}}
\newcommand{\ZZ}{\mathbb{Z}}

\newcommand{\Ss}{\mathbb{S}}
\newcommand{\CP}{\mathbb{C}P^1}

\numberwithin{equation}{section}

\newtheorem{te}{Theorem}
\newtheorem{pr}{Proposition}
\newtheorem{co}{Corollary}
\newtheorem{lm}{Lemma}

\newtheorem{de}{Definition}
\newtheorem{re}{Remark}
\newtheorem{ex}{Example}

\begin{document}

\title[Energy minimizing Beltrami fields on Sasakian 3-manifolds]{Energy minimizing Beltrami fields on Sasakian 3-manifolds}

\author{D. Peralta-Salas}
\address{Instituto de Ciencias Matem\'aticas, Consejo Superior de Investigaciones Cientificas, 28049 Madrid, Spain}
\email{dperalta@icmat.es}

\author{R. Slobodeanu}
\address{Faculty of Physics, University of Bucharest, P.O. Box Mg-11, RO--077125 Bucharest-M\u agurele, Romania}
\email{radualexandru.slobodeanu@g.unibuc.ro}

\thanks{The authors thank Vasile Br\^inz\u anescu, Paul-Andi Nagy and Mihaela Pilca for helpful advice. The first named author is supported by the ERC Starting Grant~335079, and the
ICMAT--Severo Ochoa grant SEV-2015-0554.}



\date{\today}

\begin{abstract}
We study on which compact Sasakian 3-manifolds the Reeb field, which is a Beltrami field with eigenvalue $2$, is an energy minimizer in its adjoint orbit under the action of volume preserving diffeomorphisms. This minimization property for Beltrami fields is relevant because of its connections with the phenomenon of magnetic relaxation and the hydrodynamic stability of steady Euler flows. We characterize the Sasakian manifolds where the Reeb field is a minimizer in terms of the first positive eigenvalue of the curl operator and show that for $a>a_0$ (a constant that depends on the Sasakian structure) the Reeb field of the
$\mathcal{D}$-homothetic deformation of the manifold with constant $a$ (which is still Sasakian) is an unstable critical point of the energy, and hence not even a local minimizer. We also provide some examples of Sasakian manifolds where the Reeb field is a minimizer, highlighting the case of the weighted 3-spheres, on which another minimization problem (for the quartic Skyrme-Faddeev energy) is shown to admit exact solutions.
\end{abstract}

\maketitle

\section{Introduction}\label{S:intro}
A Beltrami field on a Riemannian 3-manifold $(M,g)$ is a vector field such that:
\begin{equation}\label{belgeneq}
\cu V = f V, \quad \di V = 0,
\end{equation}
where $f$ is a smooth function on $M$, $\cu V := (* \dif V^\flat)^\sharp$ is the curl operator ($*$ denotes the Hodge star operator) and $\di$ is the divergence operator. In this paper we will be concerned with Beltrami fields for which $f$ is constant (so that the divergence-free condition becomes redundant), so in what follows when we talk about Beltrami fields, we assume that their proportionality factor is constant.

The interest in studying Beltrami fields is due to the fact that they are solutions of the stationary Euler equations which describe ideal steady fluid flows and ideal magnetohydrodynamic equilibrium configurations. In particular, Beltrami fields with a constant proportionality
factor have found application as powerful tools to analyse the Euler equations. For instance, de Lellis and Sz\'ekelyhidi have utilized
Beltrami fields to construct H\"older continuous weak solutions to the Euler
equations that dissipate energy~\cite{LS09}, while in~\cite{EP1,EP2,EP3} Beltrami fields having vortex lines and vortex tubes of arbitrary topology
were constructed. Expansions of more general solutions to the Euler or Navier-Stokes equations in terms of Beltrami fields were also considered in~\cite{Const,EP4}.

A geometric setup where Beltrami fields naturally appear is provided by contact manifolds. Indeed, by the well known Martinet's theorem any orientable closed 3-manifold $M$ admits a \textit{contact $1$-form} $\eta$ (i.e. a form that satisfies $\eta \wedge \dif \eta \neq 0$), whose \textit{contact distribution} $\Ker \eta$ is denoted by $\DD$ and its associated \textit{Reeb vector field} by $\xi$. Recall that $\xi$ is defined by the conditions $\eta(\xi)=1$ and $\dif \eta (\xi, \cdot)=0$. A Riemannian metric $g$ on $M$ is called \textit{adapted} \cite{ch} to the contact form $\eta$ if the Reeb field $\xi$ has unit norm and the volume-form $\upsilon_g$ satisfies the condition
\begin{equation}\label{dvol}
\upsilon_g = \tfrac{1}{2} \eta \wedge \dif \eta.
\end{equation}
The orientation of $M$ will be fixed by defining a $(1, 1)$-tensor field $\phi$ such that
\begin{equation}\label{conv0}
\phi^2 = -I + \eta \otimes \xi, \qquad \tfrac{1}{2}\dif \eta(\cdot, \cdot) = g(\cdot, \phi(\cdot)).
\end{equation}

An important observation is that Eq.~\eqref{dvol} together with the fact that $\dif \eta (\xi, \cdot)=0$ and $g(\xi,\xi)=1$, imply that
$$\ast \dif \eta = 2 \eta \ \ \text{or, equivalently,} \ \ \cu \xi = 2 \xi\,,$$
and therefore $\xi$ is a Beltrami field with eigenvalue 2 for the $\cu$ operator computed with the adapted metric $g$. In particular, applying curl once more, it follows that $\Delta \eta = 4 \eta$, so $\xi$ is also an eigenfield of the Hodge Laplacian. Conversely, any non-vanishing Beltrami field with non-vanishing proportionality factor is a Reeb vector field (up to a proportionality factor) for some contact 1-form~\cite{etn}.

In this paper we consider Sasakian 3-manifolds, i.e. compact contact 3-manifolds without boundary that support an adapted metric for which the Reeb vector field $\xi$ is Killing. These manifolds will be denoted here by $(M, \xi, \eta, \phi, g)$ and they correspond to $(K,1)$-manifolds in the terminology of \cite{ghri, nico}, where a $(K,\lambda)$-manifold was defined by the condition that $\xi$ is Killing and $\ast \dif \eta = 2 \lambda \eta$ for some $\lambda \in \RR$. We notice that fixing $\lambda =1$ can be always achieved by a homothety $g \mapsto \lambda^2 g$, $\eta \mapsto \lambda \eta$, so choosing the Sasakian class may be seen as fixing the scale under dilations.

It is known~\cite{Bo08} that any compact Sasakian 3-manifold has a tight contact structure and falls into one of the following diffeomorphism classes~\cite{gei}:
\begin{equation} \label{3types}
\Ss^3/\Gamma, \qquad  \mathrm{Nil}^3/\Gamma, \qquad \widetilde{SL}(2, \RR)/\Gamma\,,
\end{equation}
where $\Gamma$ is any discrete subgroup of the identity component of the isometry group of the corresponding canonical metric. As to the metric, Belgun's results~\cite{bel1, bel2} show that any possible Sasakian structure is a deformation of type I or of type II of the standard Sasakian structure on each of these spaces~\cite{BGalbook}. In particular, a natural way of deforming a Sasakian geometry is using homothetic deformations, which give rise to other Sasakian manifolds:
\begin{de} \label{d1}
A $\mathcal{D}$-\emph{homothetic deformation} with positive constant $a$ of a Sasakian structure is defined by
\begin{equation}\label{dtrans}
\xi^\prime=a^{-1}\xi, \quad \eta^\prime=a\eta, \quad \phi^\prime=\phi, \quad g^\prime=ag+a(a-1)\eta\otimes \eta.
\end{equation}
\end{de}

In this paper we are interested in how the energy (or equivalently, the $L^2$ norm) of the Reeb vector field $\xi$ of a Sasakian 3-manifold changes under the action of volume-preserving diffeomorphisms. In general, the variational problem for the energy
$$
E(X)=\frac12 \int_M |X|^2 \upsilon_g \, \quad \upsilon_g := \ \text{the Riemannian volume form},
$$
restricted to the orbit of a divergence free vector field $X$ under the group of volume-preserving diffeomorphisms $\SD (M)$ is a way of characterizing steady incompressible Euler flows~\cite{arn}. In this context we address the following:
\\

\noindent {\bf Question:} On which Sasakian manifolds the Reeb vector field is an energy minimizer in its adjoint orbit under the action of volume preserving diffeomorphisms? More precisely, we want to characterize those Sasakian manifolds where $\xi$ is a minimizer of $E$ in the sense that
$$
E(\xi_t)\geq E(\xi)
$$
for any variation of $\xi$ defined as $\xi_t=d\psi_t(\xi)$, with $\psi_t\in \SD(M)$, $\psi_0=Id_M$.
\\

This problem was stated by Ghrist and Komendarczyk in~\cite{ghri}. They conjectured (Conjecture~5.4 in their paper) that the Reeb field on any compact Sasakian 3-manifold (and on any  $(K, \lambda)$-manifold) is an energy minimizer in its adjoint orbit under the action of volume preserving diffeomorphisms. The answer is well-known on the standard round 3-sphere, i.e. the Hopf field is such an energy minimizer, however the general problem has remained wide open up to now. Apart from its mathematical interest, the physical relevance of this question lies in the fact that it is related to the phenomenon of magnetic relaxation in magnetohydrodynamics and to the hydrodynamic stability of steady Euler flows, see~\cite{arn} for details.

The first main theorem of this paper shows that, up to a $\DD$-homothetic deformation, the Reeb field of a Sasakian manifold is not an energy minimizer in its adjoint orbit under the action of volume preserving diffeomorphisms, thus providing a negative answer to the question stated above and hence disproving Ghrist and Komendarczyk conjecture:

\begin{te}\label{T:main1}
Let $(M, \xi, \eta, \phi, g)$ be a compact Sasakian 3-manifold. Then there exists a positive constant $a_0$, which depends on the Sasakian manifold, such that the Reeb field of the $\mathcal{D}$-homothetic deformation of the manifold with constant $a$ is an energy minimizer in its adjoint orbit if and only if $a\leq a_0$. In fact, for $a>a_0$ the Reeb field is an unstable critical point of the energy functional under the action of volume preserving diffeomorphisms.
\end{te}

Our second main theorem characterizes the Sasakian 3-manifolds on which the Reeb field is an energy minimizer in terms of the first positive eigenvalue of the curl operator. This generalizes the well known fact that the lower bound of the energy (``helicity bounds energy'' ~\cite{arn, ghri}) is attained by the eigenfields of the $\cu$ operator corresponding to the first positive eigenvalue. A particularly relevant class of Sasakian manifolds on which this result applies is provided by the \textit{weighted 3-spheres}, which are deformations with two integer parameters of the standard sphere yielding to a Reeb field with arbitrary linking number (for definitions see Section~\ref{S:weighted}):

\begin{te}\label{T:main2}
Let $(M, \phi, \xi, \eta, g)$ be a compact Sasakian $3$-manifold. Then the Reeb vector field $\xi$ is an energy minimizer in its adjoint orbit under the action of volume preserving diffeomorphisms if and only if the first positive eigenvalue of $\cu$ is $\mu_1=2$. In particular, the Reeb field on any weighted sphere is an energy minimizer.
\end{te}

The paper is organized as follows. In Section~\ref{S:2} we review some general facts of Sasakian manifolds and prove some properties of Beltrami fields that are implied by the Sasakian structure, obtaining some lower bounds for the first positive eigenvalue of the curl operator. The minimization property of the Reeb field is investigated in Section~\ref{S:3}, where we prove the first part of Theorem~\ref{T:main2} and establish how the existence of Beltrami fields tangent to the contact distribution affects the stability of the Reeb field. The proof of Theorem~\ref{T:main1} is presented in Section~\ref{S:qreg} for quasi-regular Sasakian manifolds and in Section~\ref{S:irreg} for irregular Sasakian manifolds. The techniques in each case are quite different, while the quasi-regular case is based on the existence of Beltrami fields tangent to the contact distribution, the irregular case exploits an approximation theorem due to Rukimbira. Finally, in Section~\ref{S:weighted} we show that the Reeb field is a minimizer on any weighted Sasakian sphere with coprime integer parameters (this is the second claim in Theorem~\ref{T:main2}), and we apply this result to identify absolute minimizers of the quartic Faddeev-Skyrme energy for static fields defined on a weighted 3-sphere, having arbitrary Hopf invariant.

All along this paper, $\Delta$ will denote the Hodge Laplacian, and when acting on a vector field $X$ we shall mean that
$\Delta X=(\Delta X^\flat)^\sharp$. In particular, the Laplacian on a scalar function $f$ is $\Delta f=-\di(\nabla f)$.

\section{Beltrami fields and Sasakian 3-manifolds}\label{S:2}

In this section we introduce some preliminary results concerning Beltrami fields in Sasakian 3-manifolds that we shall need in forthcoming sections. A special emphasis is given to the first positive eigenvalue of the $\cu$ operator, whose eigenfields are automatically solutions of the energy minimization problem presented in the Introduction. First, in subsection~\ref{S:2.1} we show that the first positive eigenvalue of the curl operator on a Sasakian $3$-manifold is equal to $2$ unless there exists an eigenfield tangent to the contact distribution, cf. Corollary~\ref{mag2}. In subsection~\ref{S:2.2} we prove Proposition~\ref{bohn}, where by a Bochner-type technique we obtain bounds in terms of the curvature for the first positive eigenvalue of the curl operator (restricted to the contact distribution).

Before establishing the aforementioned results, and for future reference, we state the structural equations of a Sasakian 3-manifold~\cite{nico} $(M, \xi, \eta, \phi, g)$ with respect to a (local) adapted orthonormal coframe $\{\eta, \omega_1, \omega_2\}$:
\begin{equation} \label{struct}
\begin{split}
\dif \eta & = 2 \omega_1 \wedge \omega_2 ,\\
\dif \omega_1 & = -(C_0 + 1)\, \eta \wedge \omega_2  - C_1 \, \omega_1 \wedge \omega_2 ,\\
\dif \omega_2 & =  \quad (C_0 + 1)\, \eta \wedge \omega_1  - C_2 \, \omega_1 \wedge \omega_2 ,\\
\end{split}
\end{equation}
where $C_i$ are functions on $M$ satisfying :
\begin{equation} \label{intcond}
\begin{split}
X_1(C_0) - \xi(C_1) - C_2(C_0+1) =& 0,\\
X_2(C_0) - \xi(C_2) + C_1(C_0+1) =& 0.
\end{split}
\end{equation}
Note that Eq.~\eqref{struct} can be rewritten as:
\begin{equation} \label{struct1}
\begin{split}
& [\xi, X_1]=-(C_0+1)X_2, \quad [X_2, \xi]=-(C_0+1)X_1, \\
& [X_1, X_2]=-2\xi+C_1 X_1 + C_2 X_2,
\end{split}
\end{equation}
where $X_1:=\omega_1^\sharp$ and $X_2:=\omega_2^\sharp$, so that~\eqref{intcond} are the integrability conditions given by the Jacobi identity (first Bianchi identity). We notice that the second Bianchi identities are automatically satisfied once~\eqref{intcond} hold true.

The $\phi$-sectional curvature is given by the (unique) function:
\begin{equation}
H=K(X_1, \phi X_1)=X_1(C_2) - X_2(C_1)-C_1^2-C_2^2 + 2C_0 - 1.
\end{equation}
We have the identity
$$
H= \tfrac{1}{2}(\Scal - 4),
$$
which is constant along $\xi$: $\xi(H)=\xi(\Scal)=0$.

The connection coefficients deduced from \eqref{struct1} read:
\begin{equation}\label{conne}
\left\{
\begin{array}{cccc}
\nabla_{\xi}\xi=0, &
\nabla_{X_1}\xi=X_2 , &
\nabla_{X_2}\xi=-X_1  \\[3mm]
\nabla_{\xi}X_1=-C_0 \, X_2 , &
\nabla_{X_1}X_1= -C_1 X_2 , &
\nabla_{X_2}X_1=\xi - C_2 \, X_2 \\[3mm]
\nabla_{\xi}X_2= C_0 \, X_1, &
\nabla_{X_1}X_2=-\xi + C_1 \, X_1, &
\nabla_{X_2}X_2 = C_2 \, X_1 \\[3mm]
\end{array}
\right.
\end{equation}

Notice that $\phi X_1 = -X_2$, $\phi X_2 = X_1$, since on any Sasakian manifold $\phi X = - \nabla_X \xi$ for every vector field $X$. Indeed, by \eqref{conv0} we have $g(X, \phi Y) = \frac{1}{2}\dif \eta(X, Y)= \frac{1}{2}\left(g(\nabla_X \xi, Y) - g(\nabla_Y \xi, X)\right)$; but, as $\xi$ is Killing, $g(\nabla_X \xi, Y)+ g(\nabla_Y \xi, X)=0$ so that $g(X, \phi Y)=- g(X, \nabla_Y \xi)$, that is $\phi Y = - \nabla_Y \xi$ for any $Y$.

Finally, let us write the components of the curl operator in the adapted orthonormal coframe $\{\eta, \omega_1, \omega_2\}$ on a Sasakian 3-manifold $M$. Let $\alpha = f\eta + f_1 \omega_1 + f_2 \omega_2$ be a 1-form expanded in such a coframe, and $X=\alpha^\sharp$ its dual vector field. Then
\begin{equation} \label{curl123}
\begin{split}
\ast \dif \alpha = & \big(X_1(f_2) - X_2(f_1)- C_1f_1 -C_2 f_2 +2f \big)\eta +\\
& \big(-\xi(f_2) + X_2(f) + (C_0+1) f_1 \big)\omega_1 + \\
& \big(\xi(f_1) - X_1(f) + (C_0+1) f_2 \big)\omega_2.
\end{split}
\end{equation}
In particular, $X \in \DD$, tangent to the contact distribution, is a $\mu$-eigenfield of $\cu$ (i.e. $\cu X = \mu X$, or equivalently $\ast \dif \alpha = \mu \alpha$) if and only if
\begin{equation} \label{deig}
\begin{split}
&\xi(f_1) = (\mu - C_0 -1) f_2, \\
&\xi(f_2) = - (\mu - C_0 -1) f_1, \\
& X_1(f_2) - X_2(f_1) - C_1 f_1 - C_2 f_2 = 0. \\
\end{split}
\end{equation}

Notice that by taking the derivative along $\xi$ of the last equation \eqref{deig} we obtain $\di X = 0$, since:
\begin{equation}\label{diverg}
\di X = X_1(f_1) + X_2(f_2) - C_2 f_1 + C_1 f_2.
\end{equation}

\subsection{The first positive eigenvalue of the curl operator}\label{S:2.1}

The following lemma establishes an identity that holds on any Riemannian manifold, and will be instrumental in order to prove Proposition~\ref{mag}.
\begin{lm}\label{movedelta}
Let $X$ be a vector field on a Riemannian manifold $(M,g)$. Then, if $Y$ is a Killing vector field on the manifold, the following identity holds:
\begin{equation*}
\begin{split}
g(\Delta X, Y) = &\Delta(g(X, Y)) + g(X, \Delta Y)-2\di(\nabla_X Y).
\end{split}
\end{equation*}
\end{lm}

\begin{proof}
We first check that for any vectors $X$ and $Y$:
\begin{equation*}
\begin{split}
g(-\tr \nabla^2 X, Y) = &\Delta(g(X, Y)) + g(X, -\tr \nabla^2 Y)\\
&+2e_i(g(X, \nabla_{e_i}Y))-2g(X, \nabla_{\nabla_{e_i}e_i}Y),
\end{split}
\end{equation*}
where $\{e_i\}$ is a local orthonormal frame. Since $Y$ is a Killing vector field, the above identity becomes:
\begin{equation*}
\begin{split}
g(-\tr \nabla^2 X, Y) = &\Delta(g(X, Y)) + g(X, -\tr \nabla^2 Y)-2\di(\nabla_X Y),
\end{split}
\end{equation*}
which, by using Weitzenb\"ock identity
\begin{equation} \label{weitz}
\Delta X = -\tr \nabla^2 X + \Ric X,
\end{equation}
yields the desired result.
\end{proof}

\begin{pr}\label{mag}
Let $(M, \xi, \eta, \phi, g)$ be a Sasakian $3$-manifold. If $X$ is a $\cu$ eigenfield,
$\cu X = \mu X$, then $f=\eta(X)$ is an eigenfunction of the Laplacian:
\begin{equation}\label{magic}
\Delta f = \mu(\mu-2)f.
\end{equation}
In particular, if $f$ does not identically vanish, then either $\mu \leq 0$ or $\mu \geq 2$.
\end{pr}

\begin{proof}
Let $(M, \xi, \eta, \phi, g)$ be a Sasakian 3-manifold and $X$ a vector field satisfying $\cu X = \mu X$ for some $\mu \in \RR$. Let $f$ denote $\eta(X)$. Applying Lemma \ref{movedelta} to $X$ and $\xi$ we obtain (recalling that $\Delta \xi =4\xi$):
\begin{equation}\label{deltas}
\mu^2 f = \Delta f + 4f + 2\di(\phi X)
\end{equation}
since $\phi X=-\nabla_X \xi$. With respect to an adapted orthonormal co-frame  $\{\xi, X_1= \phi X_2 , X_2 \}$ (so that $\dif \eta = 2\omega_1 \wedge \omega_2$), we write $\alpha =X^\flat= f \eta + f_1 \omega_1 + f_2 \omega_2$ and since $\di(\phi X)=-\delta (\phi \alpha)$ for $\phi \alpha = (\phi X)^\flat$, we compute

\begin{equation*}
\begin{split}
\delta (\phi \alpha)&=-\ast \dif \ast (\phi \alpha)  = -\ast \dif \ast (-f_1 \omega_2 + f_2 \omega_1) \\
 & = -\ast \dif (-f_1 \eta \wedge \omega_1 - f_2 \eta \wedge \omega_2) = \ast \dif (\eta \wedge (\alpha-f\eta))\\
 & = \ast (\dif \eta \wedge (\alpha-f\eta)- \eta \wedge (\dif \alpha-f \dif \eta))\\
 & = \ast (- \eta \wedge \dif \alpha + f \eta \wedge \dif \eta)\\
 & =  -(\ast \dif \alpha)(\xi) + 2f,
\end{split}
\end{equation*}
where in the last line we used the general identity for a 1-form $\theta$:
$\ast(\theta \wedge \beta)=\imath_{\theta^\sharp}\ast \beta$. Since by hypothesis
$\ast \dif \alpha = \mu \alpha$ we obtain: $\delta (\phi \alpha)=(2-\mu)f$, so $\di(\phi X)=(\mu-2)f$, which yields the result by injecting it in \eqref{deltas}.
\end{proof}

Noticing that the curl operator has a discrete spectrum on any compact Riemannian $3$-manifold without boundary \cite{baer}, and that the Reeb field satisfies $\cu \xi=2\xi$, the following corollary is an immediate consequence of Proposition~\ref{mag}:
\begin{co}\label{mag2}
Let $(M, \xi, \eta, \phi, g)$ be a Sasakian 3-manifold. If the spectrum of the $\cu$ operator acting on vector fields tangent to the contact distribution is not empty, assume that its smallest positive eigenvalue $\mu_1^\mathcal D$ satisfies $\mu_1^\mathcal D\geq 2$. Then $\mu_1=2$ is the first positive eigenvalue of
$\cu$.
\end{co}

By rewriting Equations~\eqref{deig} with respect to a $\mathcal{D}$-homothetically deformed metric we easily obtain:
\begin{lm} \label{murescale}
If $X \in \DD$ is a $\mu$-eigenfield of $\cu$, then $X$ is a $\frac{\mu}{a}$-eigenfield of $\cu$ on the $\DD$-homothetic deformation  $(M, \xi^\prime, \eta^\prime, \phi^\prime, g^\prime)$ of $(M, \xi, \eta, \phi, g)$ with constant $a>0$.
\end{lm}

Corollary~\ref{mag2} and Lemma~\ref{murescale} readily imply that the first positive eigenvalue of the curl operator on any Sasakian $3$-manifold is $2$, up to $\DD$-homothetic deformations:
\begin{pr}\label{P:imp}
Let $(M, \xi, \eta, \phi, g)$ be a compact Sasakian 3-manifold. If there is no eigenfield of $\cu$ tangent to the contact distribution, then the first positive eigenvalue of $\cu$ is $\mu_1=2$. If there are eigenfields tangent to the contact distribution, then $\mu_1=2$ on the $\mathcal D$-homothetic deformation of the manifold with constant $a$ if and only if $a\leq \frac{\mu_1^\DD}{2}$.
\end{pr}

We will eventually show that, in fact, on Sasakian manifolds there always exist eigenfields tangent to the contact distribution. Notice also that a $\mathcal{D}$-homothetic deformation entails the contraction/dilation only of the eigenvalues associated to eigenvectors in $\DD$, while the Reeb field is still an eigenfield with eigenvalue $\mu=2$, and eigenvectors with mixed components along $\xi$ and $\DD$ are generically not conserved. This feature is reminiscent of the behaviour of eigenvalues on spaces with collapsing metrics, cf.~\cite{colb}.

\subsection{Bochner-type results for Beltrami fields}\label{S:2.2}
In this subsection we investigate some general restrictions on the $\cu$ eigenvalues or eigenfields in terms of the curvature of a Sasakian 3-manifold $(M, \xi, \eta, \phi, g)$, obtaining in particular an estimate for
$\mu_1^\DD$ involved in Proposition \ref{P:imp}. We mention that Lichnerowicz-like lower bounds for $\mu$ in terms of Ricci curvature bounds (that generally do not hold in our case) have been recently obtained in \cite{baer}.

We begin by recalling a special feature of the Ricci curvature in this set-up~\cite{OGau}:
\begin{equation} \label{ric3}
\Ric = \left(\tfrac{1}{2}\Scal-1\right)g+\left(3-\tfrac{1}{2}\Scal \right)\eta \otimes \eta\,.
\end{equation}
This implies that $\xi$ is an eigenfield for $\Ric$ with constant eigenvalue $2$, while $\DD=\Ker \eta$ is an eigen-subbundle with respect to the (generally non-constant) eigenvalue $\tfrac{1}{2} \Scal-1$.

For any $X, Y \in \Gamma(TM)$, let us recall Yano's identity~\cite{yan}:
\begin{equation}\label{yano}
\begin{split}
& \mathrm{Ric}(X, Y) =\\ &= \di\left(\nabla_X Y \right)- X(\di Y) -
\tfrac{1}{2}\langle \mathcal{L}_{X}g , \mathcal{L}_{Y}g \rangle + \tfrac{1}{2}\langle\dif X^{\flat}, \dif Y^{\flat}\rangle \\
&= \di\left(\nabla_X Y \right)- X(\di Y) -
\langle \nabla X ,  \nabla Y \rangle + \langle\dif X^{\flat}, \dif Y^{\flat}\rangle \\
\end{split}
\end{equation}
where the (pointwise) metric on the bundle of $2$-covariant tensors (or $2$-forms) on $M$ is $\langle \mathcal{A}, \mathcal{B}\rangle=\frac{1}{2}\sum_{i_1, i_2=1}^{3} \mathcal{A}(e_{i_1}, e_{i_2})\mathcal{B}(e_{i_1}, e_{i_2})$.

The following proposition is the main result of this subsection. It establishes lower bounds for the first positive eigenvalue of the $\cu$ operator assuming some geometric properties of the Sasakian manifold.

\begin{pr}[Bochner type result]\label{bohn}
Let $(M, \phi, \xi, \eta, g)$ be a compact Sasakian $3$-manifold.
\begin{enumerate}
\item If $\Scal > -2$, then the first positive eigenvalue of the $\cu$ operator associated to an eigenfield $X \in \DD$ (if any) satisfies $\mu_1^\DD \geq \tfrac{1}{4}\min _M \Scal +\tfrac{1}{2}$, with equality holding if and only if both $\abs{X}$ and the scalar curvature are constant.

\item If $\Scal \geq 6$, then the first positive eigenvalue of the $\cu$ operator is $\mu_1 = 2$. In particular, the first non-zero eigenvalue of the Hodge Laplacian restricted to co-closed $1$-forms is $4$.
\end{enumerate}
\end{pr}

\begin{proof}
Item $(1)$: Let $X \in \DD$ be a $\cu$ eigenfield with eigenvalue $\mu > 0$. Choose an adapted orthonormal basis $\{\xi, X_1, X_2\}$. We can check that:

\begin{itemize}
\item $\nabla_ X X = \frac{1}{2}\gr \abs{X}^2$ (this holds for any Beltrami field). \\

\item $\nabla_{\xi} X = (\mu -1) \phi X$ so that $\abs{\nabla_{\xi} X}^2 = (\mu -1)^2 \abs{X}^2$. \\

\item $\abs{\nabla_{X_1} X}^2 + \abs{\nabla_{X_2} X}^2 =  \abs{X}^2 +  \frac{\abs{\gr \abs{X}^2}^2}{2\abs{X}^2}$ at any point $p\in M$ where $X(p)\neq0$. In fact, since the zero set of a Beltrami field is nowhere dense (by unique continuation~\cite{Kaz}), it is easy to check that the following equality also holds on the whole $M$: $\abs{\nabla_{X_1} X}^2 + \abs{\nabla_{X_2} X}^2 =  \abs{X}^2 +  2\abs{\gr \abs{X}}^2$.\\

\item $\abs{\dif X^{\flat}}^2 =  \mu^2 \abs{X}^2$.\\
\end{itemize}

Using these properties, Yano's identity and Eq.~\eqref{ric3}) we obtain
$$
\left(\tfrac{1}{2}\Scal-1\right)\abs{X}^2=-\tfrac{1}{2}\Delta \left( \abs{X}^2 \right) +2(\mu -1)\abs{X}^2
- 2\abs{\gr \abs{X}}^2\,,
$$
which can be rewritten as:
\begin{equation}\label{newbohn0}
-\tfrac{1}{2}\Delta \left( \abs{X}^2 \right)=\left(\tfrac{1}{2}\Scal +1 -2\mu\right)\abs{X}^2 + 2\abs{\gr \abs{X}}^2,
\end{equation}
or, equivalently, since $\tfrac{1}{2}\Delta f^2 = f\Delta f - \abs{\gr f}^2$ for any $f$,
\begin{equation}\label{newbohn}
-\abs{X} \, \Delta \left( \abs{X} \right)=\left(\tfrac{1}{2}\Scal +1 -2\mu\right)\abs{X}^2 + \abs{\gr \abs{X}}^2,
\end{equation}
where $\Delta = -\di \circ \gr$.

Integrating Eq.~\eqref{newbohn0} (or applying the maximum principle in \eqref{newbohn}) we obtain the lower bound
\begin{equation}\label{lowb0}
\mu \geq \tfrac{1}{4}\min _M \Scal +\tfrac{1}{2},
\end{equation}
with equality holding if and only if both $\abs{X}$ and the scalar curvature are constant.

Item $(2)$: According to Item (1), $\Scal \geq 6$ implies $\mu_1^\DD \geq 2$. Applying Corollary \ref{mag2} yields the conclusion.
\end{proof}


\begin{re}
In item $(1)$ of the above Proposition assume that $\abs{X}$ is not constant. Then, by integrating Eq.~\eqref{newbohn} and taking into account that $\xi(\abs{X})=0$, we obtain the lower bound
\begin{equation}\label{lowb}
\mu_1^\DD \geq \tfrac{1}{4}\min _M \Scal +\tfrac{1}{2} + \lambda_1(\eta).
\end{equation}
where $\lambda_1(\eta)$ is the first eigenvalue of the sub-Laplacian $\Delta+\xi^2$ acting on functions.
\end{re}

\noindent We finish this subsection by presenting two examples of Sasakian 3-manifolds, the Berger sphere and the weighted sphere, which illustrate the implications of Proposition~\ref{bohn}:

\begin{ex} \label{berger}
On the round unit $3$-sphere $\Ss^3$ with the standard Sasakian structure consider the $\DD$-homothetic deformation $g^\prime=a(g^\DD + a\eta \otimes \eta)$, which is a rescaling of the well-known \emph{Berger metric}. Then $\Scal^\prime=-2+ 8a^{-1}$, so that if $0 <a \leq 1$ we have $\Scal^\prime \geq 6$, thus $\mu \geq 2$ on $(\Ss^3, g^\prime)$ (this fact was used in~\cite{svee}).
\end{ex}

More generally, if $\Scal > -2$, then for any $0< a \leq \frac{1}{8}(\min_M \Scal + 2)$ we have $\Scal^\prime =a^{-1}(\Scal+2)-2 \geq 6$, thus $\mu \geq \mu_1 = 2$ on $(M, g^\prime)$. Then, Item $(1)$ in Proposition~\ref{bohn} tells us that this (approximate) threshold for $a$ in terms of the curvature may be lower than the (sharp) one in Proposition~\ref{P:imp}.

\begin{ex}
On the weighted sphere $\Ss^3_w$ (the definition is presented in Section~\ref{S:weighted}), since $\Scal$ is not constant, we cannot have $\cu$ $\mu$-eigenfields $X \in \DD$ of constant length. Instead, cf. \eqref{lowscalw} below we have $\min _{\Ss_w^3} \Scal=8(2\ell - k) -2$, so that $\mu_1^\DD \geq 2(2\ell -k) + \lambda_1(\eta_w)$. In particular, Proposition~\ref{bohn} implies that $\mu_1 = 2$ on $\Ss_w^3$ provided that $k \leq 2\ell-1$.
\end{ex}

\section{Characterization of Reeb fields that are energy minimizers}\label{S:3}

The purpose of this section is to investigate the minimization property of the Reeb vector field $\xi$ in a Sasakian 3-manifold. The first main result, cf. Lemma~\ref{destabxi}, shows that $\xi$ is an unstable critical point of the energy $E$ in its adjoint orbit, and hence not even a local minimizer, provided there exists an eigenfield tangent to the contact distribution with an eigenvalue less than 2. This (negative) result is complemented with Theorem~\ref{minfirst}, where we prove that $\xi$ is an energy minimizer if and only if the first positive eigenvalue of the curl operator is $\mu_1=2$, therefore reducing our energy minimization problem to a spectral geometry question.

In order to achieve these goals we need to obtain a second variation formula for the energy under the action of volume-preserving diffeomorphisms at a critical point, which turns out to be a steady incompressible Euler flow~\cite{arn}. On a compact Riemannian 3-manifold $(M, g)$ (not necessarily Sasakian), let
$\Gamma_0(TM)$ be the space of smooth divergence free vector fields. The energy functional is:
$$
E: \Gamma_0(TM)\to \RR_+, \quad E(X)=\frac{1}{2}\int_M \!\! \abs{X}^2 \upsilon_{g}.
$$
A (smooth) variation of $X$ is defined as $X_t = \dif \psi_t(X)$, where $\psi_t \in \SD_0 (M)$, $\psi_0 = Id_M$ (here $\SD_0 (M)$ denotes the identity
component of the group of volume preserving diffeomorphisms). The \textit{variation vector field} is defined as:
$$
v = \frac{\partial \psi_t}{\partial t} \Bigl \vert _{t=0} \in
\Gamma_0(TM).
$$
\begin{pr}[First variation formula] For any divergence-free vector field $X$, the following formula holds:
\begin{equation}\label{var1}
\frac{\dif}{\dif t}E(X_t)\Bigl \vert _{t=0}=-\int_M \langle v, \nabla_X X \rangle \upsilon_g .
\end{equation}
In particular, $X$ is a critical point of $E$ with respect to variations through volume preserving diffeomorphisms if and only if there exists a function $p$ on $M$ such that $\nabla_X X = -\gr p$. Therefore, $X$ is a steady solution of the Euler equations for incompressible and inviscid flows.
\end{pr}

\begin{proof} Let $\Psi:(-\epsilon,\epsilon)\times M \to M$, $\Psi(t,x) = \psi_t(x)$ be a smooth family of volume-preserving diffeomorphisms with $\psi_0= Id_M$. Denote by $\nabla^{\Psi}$ the pull-back connection on $\Psi^{-1} TM$ (see \cite{ud}). Then
\begin{equation*}
\begin{split}
\frac{\dif}{\dif t}E(X_t)\Bigl \vert _{t=0}=& \int_M \langle \nabla_{\partial_t}^{\Psi} \dif \Psi(X), \dif \Psi(X) \rangle \upsilon_g \Bigl \vert _{t=0} \\
(\textrm{since} \ [X, \partial_t]=0) \quad =& \int_M \langle \nabla_{X}^{\Psi} \dif \Psi(\partial_t), \dif \Psi(X) \rangle \upsilon_g \Bigl \vert _{t=0} \\
=&  - \int_M \langle \dif \Psi(\partial_t), \nabla_{X}^{\Psi} \dif \Psi(X) \rangle \upsilon_g \Bigl \vert _{t=0} = - \int_M \langle v, \nabla_{X} X \rangle \upsilon_g,
\end{split}
\end{equation*}
where in the third equality we used $\di X = 0$ (so that $\int_M X(f)\upsilon_g=0$ for any scalar function $f$). The result follows from Hodge decomposition.
\end{proof}

Recall that any Beltrami field  $X$ (solution of \eqref{belgeneq}) is in particular a critical point for the above variational problem, with $p=-\frac{1}{2}\abs{X}^2$ (up to an additive constant).

We are now ready to compute the second variation of $E$ at a critical point:

\begin{pr}[Second variation formula]\label{P:secondvar}
Let $X$ be a steady incompressible Euler flow with pressure $p$, hence a critical point of the energy in its $\SD(M)$ orbit, and $\{X_t\}$ be a variation through volume preserving diffeomorphisms with variation vector $v$. Then:
\begin{equation}\label{var2}
\frac{\dif^2}{\dif t^2}E(X_t)\Bigl \vert _{t=0} =
-\int_M \{\mathrm{Hess}_p(v,v)+ \langle v, \nabla_X \nabla_X v + R(v, X)X \rangle\} \upsilon_g.
\end{equation}
\end{pr}

\begin{proof} Let $\Psi:(-\epsilon,\epsilon)\times M \to M$, $\Psi(t,x) = \psi_t(x)$ be, as before, the smooth deformation of $Id_M$ in $\SD_0(M)$ tangent to $v$ that gives the variation of $X$. By the first variation formula,
$$\frac{\dif}{\dif t}E(X_{t}) = - \int_M \langle \dif \Psi(\partial_t), \nabla_{X}^{\Psi} \dif \Psi(X) \rangle \upsilon_g.$$
Differentiation of this with respect to $t$ gives
\begin{equation*}
\begin{split}
\frac{\dif^2}{\dif t^2}E(X_t)= &- \int_M \frac{\dif}{\dif t}\langle  \dif \Psi(\partial_t), \nabla_{X}^{\Psi} \dif \Psi(X) \rangle \upsilon_g \\
= &- \int_M \left(\left\langle  \nabla_{\partial_t}^{\Psi} \dif \Psi(\partial_t), \nabla_{X}^{\Psi} \dif \Psi(X) \right\rangle + \left\langle  \dif \Psi(\partial_t), \nabla_{\partial_t}^{\Psi} \nabla_{X}^{\Psi} \dif \Psi(X) \right\rangle \right) \upsilon_g.
\end{split}
\end{equation*}
Let us compute the first integral term in the right-hand side of the above equation. Without loss of generality we may choose $\psi_t$ to be the flow of the divergence-free vector $v$ (i.e. $\dif \Psi(\partial_t)=v\circ \Psi$). Using the properties of the pull-back connection \cite{ud}, this term evaluated at $t = 0$ is $-\int_M \left \langle \nabla_v v, \nabla_{X}X \right\rangle \upsilon_g = \int_M \left \langle v, \nabla_v\nabla_{X}X \right\rangle \upsilon_g = -\int_M \Hess_p(v,v) \upsilon_g$ since $\nabla_{X} X=-\gr p$ by hypothesis and where we used again the fact that $v$ is divergence free.

As to the second summand, we have
\begin{equation*}
\begin{split}
\nabla_{\partial_t}^{\Psi} \nabla_{X}^{\Psi} \dif \Psi(X)=&  \nabla_{X}^{\Psi} \nabla_{\partial_t}^{\Psi} \dif \Psi(X)+ R(\dif \Psi (\partial_t), \dif \Psi (X))\dif \Psi(X)\\
(\textrm{since} \ [X, \partial_t]=0) =&  \nabla_{X}^{\Psi} \nabla_{X}^{\Psi} \dif \Psi(\partial_t)+ R(\dif \Psi (\partial_t), \dif \Psi (X))\dif \Psi(X)\\
(\textrm{evaluate at} \ t=0) \ =& \nabla_{X}\nabla_{X} v + R(v, X)X\,.
\end{split}
\end{equation*}
Combining with the computation of the first term, this yields the result.
\end{proof}

As usual, a critical point $X$ of our variational problem is called a \emph{local minimizer} of the energy or a \textit{stable solution} if $\frac{\dif^2}{\dif t^2}E(X_t)\bigl \vert _{t=0} \geq 0$.

\begin{re}
Using that $X$ is divergence-free, Eq.~\eqref{var2} is equivalent to:
$$\frac{\dif^2}{\dif t^2}E(X_t)\Bigl \vert _{t=0} =
\int_M \!\! \left\{\abs{\nabla_X v}^2 - K(v, X)(\abs{v}^2 \abs{X}^2-\langle v, X \rangle^2)-\mathrm{Hess}_p(v,v)\right\} \upsilon_g.$$
In particular, any steady Euler flow with constant pressure in a compact manifold of negative sectional curvature is a local minimizer of $E$.
\end{re}

In the rest of the section we shall apply the above general results to study the stability of the Reeb vector field $\xi$ on a Sasakian 3-manifold. First, we prove that in the presence of Beltrami fields tangent to the contact distribution, $\xi$ is not a local minimizer under appropriate $\DD$-homothetic rescalings:

\begin{lm} \label{destabxi}
Let $(M, \phi, \xi, \eta, g)$ be a compact Sasakian $3$-manifold. If there exists a curl eigenfield $v$ with $\mu>0$ tangent to the contact distribution, then the second variation of the energy of $\xi$ along $v$ reads:
$$\frac{\dif^2}{\dif t^2}E(\xi_t)\bigl \vert _{t=0} = \mu (\mu -2)\|v\|_{L^2}^2\,.$$
In particular, if $\mu < 2$ (that can always be achieved after a $\DD$-homothetic rescaling), then the Reeb vector field is an unstable critical point of the energy functional (and hence not a local minimizer).
\end{lm}

\begin{proof}
As $\xi$ is a geodesic (and Beltrami) vector field, it is obviously a critical point of the energy in its $\SD(M)$ orbit, with constant pressure $p$. Using the notation introduced in Section 2, let $v=f_1 X_1 + f_2 X_2$ be a $\mu$-eigenfield of $\cu$. Then, from Eq.~\eqref{deig} and the structure equations we obtain $\nabla_{\xi} v = (\mu -1) (f_2 X_1 - f_1 X_2)$. On the other hand, since on any Sasakian manifold the identity $R(X,\xi)Y=\eta(Y)X- g(X,Y)\xi$ holds for any $X,Y$, we see that $R(v,\xi)\xi = v$. Accordingly, the second variation of the energy \eqref{var2} at $\xi$ for a 1-parameter variation tangent to $v$ reads:
$\frac{\dif^2}{\dif t^2}E(\xi_t)\bigl \vert _{t=0} = \mu (\mu -2)\|v\|_{L^2}^2$. Therefore $\frac{\dif^2}{\dif t^2}E(\xi_t)\bigl \vert _{t=0} <0$, and $\xi$ cannot be a local minimizer if $0< \mu <2$. Finally, cf. Lemma~\ref{murescale}, one can always choose a $\DD$-homothetic rescaling of the metric such that $\mu <2$.
\end{proof}

In the following theorem we characterize the Sasakian manifolds such that $\xi$ is a minimizer of the energy as those manifolds where the first positive eigenvalue of curl is $\mu_1=2$. Together with Lemma~\ref{destabxi}, this reduces the Question stated in Section~\ref{S:intro} to studying the eigenvalues of the curl operator and the existence of eigenfields tangent to the contact distribution. This will be key for the results we obtain in forthcoming sections. In particular, this proves the first claim of Theorem~\ref{T:main2} stated in the Introduction.

\begin{te}\label{minfirst}
 Let $(M, \phi, \xi, \eta, g)$ be a compact Sasakian $3$-manifold. Then the Reeb vector field $\xi$ is a minimizer of the energy in its $\SD(M)$-orbit if and only if the first positive eigenvalue of $\cu$ is $\mu_1=2$.
\end{te}

\begin{proof}  The fist implication is a direct consequence of the fact~\cite{arn, ghri} that the energy $E$ of an exact\footnote{meaning that $\imath_X \upsilon_g$ is an exact 2-form} divergence-free vector field $X$ is lower bounded by $\frac{1}{2}\mu_1\big(\cu^{-1}X, X\big)_{L^2}$ where the $L^2$ inner product term is called the \textit{helicity of} $X$ and is an invariant of the $\SD(M)$-orbit of $X$~\cite{arn,PS}. Indeed, suppose that the first positive eigenvalue of $\cu$ is $\mu_1=2$. Let $T$ be a vector field in the adjoint orbit of $\xi$ under the action of $\SD(M)$. As is well known~\cite{arn, PS} this implies that the helicities of $\xi$ and $T$ are the same. Since $(\ast \dif)^{-1}$ is a compact operator defined on the space of co-closed 1-forms $L^2$-orthogonal to the kernel of $\ast \dif$, there exists a basis of eigenforms $\{\alpha_I \}_{I\in \ZZ}$: $(\ast \dif)^{-1}\alpha_I = \frac{1}{\mu_I}\alpha_I$, such that $...\leq \mu_{-2}\leq \mu_{-1} < 0 < \mu_{1} \leq \mu_{2} \leq ...$.  Expanding $T^\flat=\theta=\sum_I c_I\alpha_I$, we get (recall that the energy is half the $L^2$-norm)
\begin{equation*}
\begin{split}
& 0 < E(\xi)=\left((\ast \dif)^{-1} \eta, \eta \right)_{L^2}=\left((\ast \dif)^{-1} \theta, \theta \right)_{L^2} = \sum_{i \geq 1} \frac{c_i^2}{\mu_i}+\sum_{i \geq 1} \frac{c_{-i}^2}{\mu_{-i}} \\
& \leq \sum_{i \geq 1} \frac{c_i^2}{\mu_i} \leq \frac{1}{\mu_1}\sum_{i \geq 1} c_i^2 \leq \frac{1}{\mu_1}\norm{\theta}_{L^2}^2 = E(T)\,,
\end{split}
\end{equation*}
where in the last equality we have used that $\mu_1=2$. This implies that $\xi$ is indeed a minimizer in its $\SD(M)$-orbit.

Conversely, assume that $\xi$ is a minimizer in its $\SD(M)$-orbit. Let $\mu>0$ be an eigenvalue of $\cu$ with $X$ an associated eigenfield. If $X \in \DD$, then according to Lemma~\ref{destabxi} it follows that $\mu \geq 2$ since $\xi$ is a minimizer.  If $X \notin \DD$, then Proposition \ref{mag} applies showing also that $\mu \geq 2$.
\end{proof}

This theorem and Proposition~\ref{P:imp} easily imply the following corollary. The statement is the same as Theorem~\ref{T:main1} up to the existence of curl eigenfields tangent to the contact distribution.

\begin{co}\label{co:main}
Let $(M, \xi, \eta, \phi, g)$ be a compact Sasakian $3$-manifold. Assume that the positive spectrum of the curl operator acting on fields tangent to the contact distribution is non-empty, and let us denote the smallest eigenvalue of such a set by $\mu_1^\DD$. Then the Reeb field of the $\mathcal{D}$-homothetic deformation of the manifold with constant $a$ is an energy minimizer in its adjoint orbit if and only if $a\leq \frac{\mu^\DD_1}{2}$. In fact, for $a>\frac{\mu_1^\DD}{2}$ the Reeb field is an unstable critical point of the energy functional under the action of volume preserving diffeomorphisms.
\end{co}

\section{Proof of Theorem~\ref{T:main1} for quasi-regular Sasakian 3-manifolds}\label{S:qreg}

We recall that a Sasakian manifold is quasi-regular if the flow of the Reeb field induces a locally free $\mathbb S^1$-action; in particular all orbits  are compact. If the action is free then the Sasakian manifold is called \textit{regular}. This implies that $M$ is the total space of a $\mathbb S^1$-bundle over a 2-orbifold (or a Riemann surface in the regular case).

In this section we prove Theorem~\ref{T:main1} for Sasakian 3-manifolds that are \textit{quasi-regular}. To this end, in view of Theorem~\ref{minfirst}, we have to study the first positive eigenvalue $\mu_1$ of $\cu$. Since the eigenvalue of a Beltrami field that is not tangent to the contact distribution is necessarily $\geq 2$ (if positive), cf. Proposition~\ref{mag}, it is enough to fix our attention on curl eigenfields tangent to the contact distribution. If such fields exist, then Corollary~\ref{co:main} implies that for any $a>a_0:=\frac{\mu_1^\DD}{2}$, the Reeb field of the $\mathcal{D}$-homothetic deformation of the manifold with constant $a$ is unstable, and hence not an energy minimizer, while it is a minimizer provided that $a\leq a_0$. We mention that the related problem of describing the spectrum and the eigenspaces of the Hodge Laplacian on the total space of a circle bundle over a Hodge manifold has been discussed in~\cite{nag0, nag}.

The main result of this section is to establish that the aforementioned Beltrami fields tangent to the contact distribution indeed exist on any quasi-regular Sasakian 3-manifold:
\begin{te}\label{genexist}
Let $(M, \xi, \eta, \phi, g)$ be a compact quasi-regular Sasakian 3-manifold. Then there exists a sequence $\{\mu_k\}\nearrow \infty$ of positive eigenvalues of the curl operator with associated (nontrivial) $\mu_k$-eigenfields tangent to the contact distribution.
\end{te}

Before proving this result we need to show some preliminary properties. Recall that eigenfields tangent to the contact distribution are characterized by Eqs.~\eqref{deig}, which easily yields the following:
\begin{lm}\label{phix}
If $X \in \DD$ is a $\mu$-eigenfield of $\cu$, then so is $\phi X$.
\end{lm}

This simple remark opens the way of introducing holomorphic ideas into the play. We first recall some terminology. Given a Sasakian manifold $(M, \phi, \xi, \eta, g)$, the complexified tangent bundle admits a natural splitting:
$$
T^{\mathbb{C}}M=T^{0}M \oplus T^{(1,0)}M \oplus T^{(0,1)}M\,,
$$
where $T^{(1,0)}M=\{X-\ii \phi X \mid X \in \Gamma(\mathcal{D})\}$, $T^{(0,1)}M=\ov{T^{(1,0)}M}$
and $T^{0}M=Span_{\mathbb{C}}\{\xi\}$ are the eigenspaces of $\phi$ corresponding to
the eigenvalues $\ii, -\ii$ and $0$, respectively. Accordingly a complex valued $p$-form $\omega$ on $M$ is a $(p,0)$-form if $\imath_{T^{(0,1)}M}\omega =0$, and similarly a complex valued $q$-form is a $(0,q)$-form if $\imath_{T^{(1,0)}M}\omega =0$ and $\imath_{\xi}\omega =0$; the corresponding bundles will be denoted by $\Lambda^{p,0} M$ and $\Lambda^{0,q} M$, and their wedge product by $\Lambda^{p,q} M$. We add the index $\DD$ when considering only \emph{horizontal forms}, i.e. $\imath_{\xi}\omega =0$. The differential operator restricted to $\Lambda_{\DD}^{p,q} M$ admits the decomposition $d = \partial + \ov{\partial} \ (\textrm{mod} \ \eta)$ into its $\Lambda_{\DD}^{p+1,q} M$ and  $\Lambda_{\DD}^{p,q+1} M$ components. Finally, recall that $K_M = \Lambda^{(\dim M+1)/2,0} M$ is the \textit{canonical CR bundle} of $M$. Since the dimension of $M$ is 3, $K_M = \eta \wedge \Lambda_{\DD}^{1,0}M$ is the bundle of 2-forms vanishing on  $T^{(0,1)}M$; the closed sections in $K_M$ are called holomorphic (2,0)-forms \cite{biq}. The following proposition is key in order to prove Theorem~\ref{genexist}. It characterizes curl eigenfields tangent to the contact distribution in terms of complex-valued vector fields or forms:
\begin{pr}\label{eigencomplex}
Let $(M, \phi, \xi, \eta, g)$ be a compact Sasakian 3-manifold and $X \in \DD$ be a vector field tangent to the contact distribution. Let $\alpha = X^\flat$ be its dual 1-form. Consider the complex vector field $Z = X - \ii \phi X \in T^{(1,0)}M$ and the $(1,0)$-form $\omega:= \alpha - \ii \alpha \circ \phi \in \Lambda_{\DD}^{1,0}M$, which is the dual of $\ov Z$. Then the following statements are equivalent:
\begin{enumerate}
\item $X$ is a $\mu$-eigenfield of $\cu$.
\item The complex vector field $Z$ satisfies
\begin{equation} \label{deigcomplex}
\begin{split}
&\nabla_\xi \ov{Z}=-\ii(\mu -1)\ov{Z}\,, \\
&\nabla_{\ov{W}} \ov{Z}= 0, \quad \forall W \in T^{(1,0)}M\,. \\
\end{split}
\end{equation}

\item The $(1,0)$-form $\omega$ satisfies:
\begin{equation}\label{domeg}
d\omega = - \ii \mu \, \eta \wedge \omega.
\end{equation}

\item The $(1,0)$-form $\omega$ satisfies:
\begin{equation}  \label{deigcomplex1}
\begin{split}
&\mathcal{L}_\xi \omega = - \ii \mu \, \omega\,, \\
&\ov{\partial} \omega = 0\,. \\
\end{split}
\end{equation}

\item The $2$-form $\varpi = \eta \wedge \omega$ is a holomorphic form in the canonical bundle $K_M$ such that \ $\ii \mathcal{L}_\xi\varpi=\mu \varpi$.
\end{enumerate}
\end{pr}

\begin{proof}
$(1) \Leftrightarrow (2)$: Recall that $X$ is a $\mu$-eigenfield of $\cu$ if and only if Eqs.~\eqref{deig} are satisfied. By expanding the derivatives in $\nabla_\xi \ov{Z} = \nabla_\xi (f_1X_1+f_2X_2) + \ii \nabla_\xi (f_2X_1-f_1X_2)$ and using Eq.~\eqref{conne}, we obtain
\begin{equation*}
\begin{split}
\nabla_\xi \ov{Z}=&(\xi(f_1)+C_0 f_2)X_1+(\xi(f_2)-C_0 f_1)X_2\\
&+\ii\left[(\xi(f_2)-C_0 f_1)X_1 - (\xi(f_1)+C_0 f_2)X_2\right]\,,
\end{split}
\end{equation*}
which clearly shows the equivalence between the first two equations in \eqref{deig} and the first equation in \eqref{deigcomplex}. Performing analogous computations one can see that the last equation in~\eqref{deig} and Eq.~\eqref{diverg} are equivalent to the second equation in~\eqref{deigcomplex} (as $\dim_\CC T^{(1,0)}M  =1$ it is enough to take $W = X_1 + \ii X_2$).

$(2) \Leftrightarrow (3)$: With respect to an adapted frame $\alpha = f_1 \omega_1 + f_2 \omega_2$, and therefore $\ast \alpha = -f_1 \eta \wedge \omega_2 + f_2 \eta \wedge \omega_1 = - \eta \wedge \alpha \circ \phi$. But $X$ is an eigenfield of $\cu$ tangent to the contact distribution if and only if $d\alpha= \mu \ast \alpha$, that is $d\alpha= -\mu \eta \wedge (\alpha \circ \phi)$. Applying Lemma \ref{phix}, we obtain easily \eqref{domeg} in terms of the associated complex 1-form $\omega$.

$(3) \Leftrightarrow (4)$: Taking into account that $\mathcal{L}_\xi \omega = \imath_\xi d\omega$  and $d = \ov{\partial}  \ (\textrm{mod} \ \eta)$ on  $\Lambda_{\DD}^{1,0} M$, we see that Equations \eqref{deigcomplex1} follow from \eqref{domeg}. The converse is immediate.

Alternatively, we can check directly the equivalence $(2) \Leftrightarrow (4)$: The first equation in~\eqref{deigcomplex} is equivalent to
$[\xi, \ov{Z}]=-\ii\mu\ov{Z}$, which in turn is equivalent to the first equation in~\eqref{deigcomplex1} as we can easily check by using the fact that $\xi$ is Killing. The second equation $\ov{\partial} \omega = 0$ is equivalent, in our context, to $d\omega(W, \ov{W})=0$ for any  $W \in T^{(1,0)}M$, or equivalently $g(\nabla_{W} \ov{Z}, \ov{W})- g(\nabla_{\ov{W}} \ov{Z}, W)=0$. Since on a Sasakian manifold  $\nabla_{W} \ov{Z} \in T^{(0,1)}M \oplus T^{0}M$, the first term is always zero, so $\ov{\partial} \omega = 0$ is equivalent with $g(\nabla_{\ov{W}} \ov{Z}, W)=0$ which in turn is equivalent with the second equation in~\eqref{deigcomplex} as we can directly check (assume $W = X_1 + \ii X_2$ as $\dim_\CC T^{(1,0)}M  =1$, then  expand the derivatives of $\ov{Z}$ written in the adapted frame by using \eqref{conne}).

$(3) \Leftrightarrow (5)$: Assume that (5) holds. First notice that $\mathcal{L}_\xi\varpi=\eta \wedge \mathcal{L}_\xi\omega$ so that the condition on $\mathcal{L}_\xi\varpi$ easily translates into $\mathcal{L}_\xi \omega = - \ii \mu \, \omega$. Second, remark that the holomorphicity condition $d\varpi = 0$ is equivalent to $\eta \wedge d\omega = 0$, which implies $d\omega = \eta \wedge \Theta$ for some \textit{horizontal} 1-form $\Theta$ (i.e. $\imath_\xi \Theta=0$). Contracting with $\xi$ the latter equation enables us to identify $\Theta$ through the formula $\Theta=\mathcal{L}_\xi\omega = -\ii \mu \omega$, and therefore we obtain  Eq.~\eqref{domeg}. The converse can be obtained along the same lines.
\end{proof}

We are now ready to prove Theorem~\ref{genexist}. Since the Sasakian manifold is assumed to be quasi-regular, the space of orbits of the Reeb flow is a 2-dimensional orbifold $\Sigma_g$, i.e. a Riemann surface of genus $g$ and $N$ marked points of integer weights $a_1, \dots, a_N$, and $M \to \Sigma_g$ is a Seifert fibration. We can also see $M$ as the circle $V$-bundle $\mathcal{S}(L)$ of a complex line $V$-bundle $L$ over $\Sigma_g$, whose structure around orbifold points is characterized by the integers $b_1, \dots, b_N$, $0 \leq b_j < a_j$ with $\mathrm{gcd}(a_j, b_j)=1$. In particular, if $\deg(L)$ is the (integer) Chern number of the smoothened line bundle $\abs{L}$, the (rational) Chern number of the line bundle $L$ is given by $c_1(L)=\deg(L) + \sum_{j=1}^N \frac{b_j}{a_j}$.  We summarize the Seifert invariants of $M$ as $\{\deg(L), g; \ (a_j, b_j)\}$. Noticing that the sign of $c_1(L)$ changes taking the conjugate Sasakian structure $(-\xi,-\eta,-\phi,g)$, and this does not affect neither the orientation of the manifold (which is fixed by $\eta\wedge \dif\eta = 2\upsilon_g$ and by $\phi$) nor the spectrum and eigenfields of the curl operator, we can safely assume that (see also the Appendix):
$$
c_1(L)<0\,.
$$

\begin{proof}[Proof of Theorem~\ref{genexist}]
The existence of eigenfields tangent to the contact distribution is not affected by $\mathcal{D}$-homothetic deformations, c.f. Lemma~\ref{murescale}. Accordingly, since $M$ is a circle bundle over an orbifold, we can safely assume, possibly up to a $\mathcal D$-homothetic deformation, that the length of a typical Reeb orbit is $2\pi$. We claim that the eigenvalue $\mu$ of an eigenfield tangent to the contact distribution is an integer. To this end we use the characterization of eigenfields tangent to the contact distribution given in item (4) of Proposition~\ref{eigencomplex}. It states that the $(1,0)$-form $\omega$ satisfies the eigenvalue problem $\mathcal{L}_\xi\omega=-\ii\mu\omega$, which implies that for any vector field $W\in T^{\mathbb C} M$ we have:
$$
\frac{d}{dt}\omega(\dif\varphi_{t}W)=-\ii\mu \, \omega(\dif\varphi_{t}W)\,,
$$
where $\varphi_t$ is the flow defined by the Reeb field $\xi$. In a neighbourhood $U$ of a typical orbit of $\xi$, one can choose the vector field $W$ such that the complex-valued function $F:[0,2\pi] \to \CC$, $F(t)=\omega(\dif \varphi_{t}W_x)$, is not identically constant for some $x \in U$ (otherwise, $\omega$ would be zero in such a neighborhood, but this is not possible because the zero set of a Beltrami field is nowhere dense). Integrating the ODE above we get
$$
\omega(\dif \varphi_{t}W_x)=\exp(-\ii\mu t)\omega(W_x)\,,
$$
thus implying that $\mu \in \ZZ$, as we wanted to show.

We now use the characterization from item (5) of Proposition~\ref{eigencomplex}. As in ~\cite[p. 619]{biq} (see the Appendix for details), the holomorphic 2-form $\varpi$ with the equivariance property $i\mathcal{L}_\xi\varpi=\mu \varpi$ can be understood as a holomorphic section of $K_{\Sigma_g}\otimes L^{-\mu}$, so, by Kodaira-Serre duality, as an element in $H^1(\Sigma_g, \Oo(L^\mu))$.  Therefore the existence of curl eigenfields tangent to the contact distribution (for an infinite sequence of eigenvalues) will follow if one is able to prove that the latter space has non-zero dimension. Indeed, according to the Riemann-Roch formula for orbifolds~\cite{kawa},
\begin{equation}
\dim_\CC H^0(\Sigma_g, \Oo(L^\mu)) - \dim_\CC H^1(\Sigma_g, \Oo(L^\mu)) = \deg L^\mu + 1 -g\,.
\end{equation}
Using Kodaira vanishing theorem for orbifolds \cite{bai}, we have $\dim_\CC H^0(\Sigma_g, \Oo(L^\mu))=0$ due to the fact that $L$ is negative and so is $L^\mu$ ($\mu >0$). Since
$\deg L^\mu = c_1(L^\mu) - \sum_{j=1}^N \frac{b_j(L^\mu)}{a_j} = \mu c_1(L)- \sum_{j=1}^N \frac{b_j(L^\mu)}{a_j}$, we obtain
\begin{equation}\label{dimh1}
\dim_\CC H^1(\Sigma_g, \Oo(L^\mu))=-1 + g - \mu c_1(L) + \sum_{j=1}^N \frac{b_j(L^\mu)}{a_j}\,,
\end{equation}
which is strictly positive for large enough $\mu$ (recall that by convention $-c_1(L) \in \mathbb{Q}_+$ and $0 \leq b_j(L^\mu)<a_j$). The theorem then follows.
\end{proof}

To finish this section, we illustrate the proof of Theorem~\ref{genexist} with several examples of
(quasi-) regular Sasakian structures for which we establish the existence of eigenfields tangent to the contact distribution, we determine $\mu_1^\DD$ and then we decide if the Reeb field is an energy minimizer by applying Theorem~\ref{minfirst}. In all these examples we choose a metric such that the length of the (regular) orbits of the Reeb field\footnote{If the period of the typical orbit of $\xi$ is $T$, then the period of the typical orbit of $a^{-1}\xi$ is $aT$. So one can always adjust the period of the fibres by a $\DD$-homothetic deformation of the metric.} is $2\pi$, in order to assure that $\mu \in \ZZ$ for the straightforward application of Eq.~\eqref{dimh1}.

We begin with the three standard Sasakian (regular) geometries: the unit round 3-sphere
$\Ss^3$, the  compact Heisenberg nilmanifolds $\mathrm{Nil}^3/\Gamma_r$, and the $\widetilde{SL}_2$-type manifolds.

\begin{ex}[Positive case] Let $M=\Ss^3$ be endowed with the standard Sasakian structure~\cite[p.211]{BGalbook}. We describe $M$ as a circle bundle over
$\CC P^1$ of degree $-1$ (Hopf fibration) with associated complex line bundle denoted by $L$. Eq.~\eqref{dimh1} with $\deg L = c_1(L) = -1$ and $g=0$ yields
$$
 \dim_\CC H^1(\CC P^1, \Oo(L^{\mu}))= -1 + \mu =\left\{
\begin{array}{ccc}
0 &,& \text{if} \ \mu=1\\[3mm]
\geq 1 &,& \text{if} \ \mu\geq 2
\end{array}
\right.
$$
Thus, the first eigenvalue associated to an eigenfield of the curl operator tangent to the contact distribution is $\mu_1^\DD=2$. Cf. Corollary \ref{mag2} we then have $\mu_1 = 2$ and the Reeb vector field $\xi$ is a minimizer of the energy in its adjoint orbit (already known).
\end{ex}

\begin{ex}[Null case] The (3-dim.) Heisenberg group
$\mathrm{Nil}^3$ is defined as the Lie group of all matrices with real entries $\begin{pmatrix} 1 & x & z \\ 0 & 1 & y \\ 0 & 0 & 1 \end{pmatrix}$ on which the subgroup with integer entries $\Gamma=\ZZ^3$ acts by left multiplication. Then the (compact) 3-dimensional Heisenberg nilmanifold is defined as $M=\mathrm{Nil}^3/\Gamma$. A (regular) Reeb field on $M$ is $\xi=\frac{1}{2\pi}\partial_z$ with contact form $\eta = 2\pi(\dif z-x \dif y)$  and the adapted Sasakian metric is $g=\pi(\dif x^2+ \dif y^2)+4\pi^2(\dif z-x \dif y)^2$.

As circle bundle over a torus $\mathbb{T}^2$, $M$ is of degree  $-1$, with associated complex line bundle $L$. Again Eq.~\eqref{dimh1} with $\deg L = c_1(L) = -1$ and $g=1$ yields $\dim_\CC H^1(\mathbb{T}^2, \Oo(L^{\mu}))= \mu \geq 1$. Therefore $\mu_1^\DD=1$, so $\mu_1 < 2$ and the Reeb vector field $\xi$ is not a minimizer of the energy in its adjoint orbit.

A similar discussion can be made for compact Heisenberg manifolds
$M_k=\mathrm{Nil}^3/\Gamma_k$ which are circle bundle of degree  $-k$ over a flat torus.
\end{ex}

\begin{ex}[Negative case]  Let $M = T_1 \Sigma_g$ be the unit tangent bundle of a compact oriented surface $\Sigma_g$ of genus $g > 1$, and with constant $-2$ Gaussian curvature metric. As $\Sigma_g = \Gamma \setminus \mathbb{H}$ with $\Gamma$ a cocompact Fuchsian group
acting on the upper half-plane $\mathbb{H}$, the manifold $M$ can be identified with $\Gamma \setminus PSL(2, \RR)$, where $PSL(2, \RR)=SL(2, \RR)/\{\pm I\}$ is itself a circle bundle over $\mathbb{H}$.  On the covering $SL(2, \RR)$ we consider the (regular) contact structure  $\eta = 2\dif \theta + \frac{1}{y} \dif x$, $\xi=\frac{1}{2}\partial_\theta$ and the adapted Sasakian metric $g=\frac{1}{2y^2}(\dif x^2+ \dif y^2)+(2\dif \theta + \frac{1}{y} \dif x)^2$, where the coordinates $(x,y,\theta)\in \RR\times (0, \infty) \times [0, 2\pi]$ are given by the Iwasawa decomposition of matrices on $SL(2, \RR)$ (notice that, when projected on $M$, $\xi$ has orbits of length $2 \pi$, like the infinitesimal generator $\partial_\theta$ of the $\Ss^1$-action on $SL(2, \RR)$).

From the Gauss-Bonnet theorem we deduce that $\mathrm{area}(\Sigma_g)=2\pi(g-1)$ and $\deg L = c_1(L) = 2(1-g)$, as expected since $L \cong T^{(1,0)}\Sigma_g$. As on the covering space we have $\mu \in \ZZ$, the same will be true on $M$ and then Eq.~\eqref{dimh1} yields
$$
\dim_\CC H^1(\Sigma_g, \Oo(L^{\mu}))= -1 + g - \mu (2-2g) = (g-1)(2\mu + 1) > 1\,,
$$
for any $\mu\geq 1$. Notice that if $\mu=1$, then the corresponding 1-forms $\omega \in K_{\Sigma_g}\otimes L^{-1} \cong K_{\Sigma_g}^2$ are holomorphic quadratic differentials forming a vector space of complex dimension $3(g-1)$, cf. e.g. \cite[Corollary 5.4.2]{jo}. As $\mu_1^\DD =1$, $\mu_1 <2$ and therefore the Reeb vector field $\xi$ is not an energy minimizer in its adjoint orbit. 
\end{ex}

We mention that in the above examples the coefficients $f_1$ and $f_2$ with respect to a global adapted frame (see \cite{ch} for details) of a curl eigenvector tangent to $\DD$ are simultaneously eigenfunctions of the Laplacian and of the vertical Laplacian. In particular the conclusions in the positive and null case can be drawn from the explicit knowledge of the (vertical) Laplacian spectrum on $\Ss^3$~\cite{lo} and $\mathrm{Nil}^3/\Gamma$~\cite{gowi}.

We continue with two quasi-regular Sasakian structures on spherical space forms. A particularly interesting quasi-regular example (the weighted sphere) will be analysed in detail in Section~\ref{S:weighted}. Since the spectrum of the curl operator on a quotient of $\mathbb S^3$ is contained in the spectrum of $\mathbb S^3$, it is obvious that $\mu_1=2$. The goal of the examples below is to show that the $\mu=2$ eigenspace contains a 2-dimensional basis of eigenfields tangent to the contact distribution using the Riemann-Roch argument introduced in the proof of Theorem~\ref{genexist}.

\begin{ex}[Real projective space and lens spaces]
The spectrum of the $\cu$ operator on $(\RR P^3, g=can)$ was calculated in~\cite{baer}: $\mu = \pm(2+k)$, with $k$ even, and with the same multiplicity as on $\Ss^3$, $m(\cu , \mu)= (k+1)(k+3)$. $\RR P^3$ with the standard structure is an example of \textit{regular} Sasakian manifold: all trajectories of the Reeb field $\xi$ are circles with the same period $\pi$. We consider a $\DD$-homothetic deformation that render the trajectories of the Reeb field of length $2\pi$: if $(\eta, \xi, \phi, g)$ is the standard structure, consider $g'=ag^\DD + a^2 \eta \otimes \eta$, $\xi'=a^{-1}\xi$ with $a=2$. Then with respect to $g^\prime$ we have $\mu^\prime=\mu/2 \in \ZZ$ (so $\mu$ have to be even) and the Riemann-Roch argument implies that $\dim_\CC H^1(\CC P^1, \Oo(L^{\mu'})) = -1 + 2\mu^\prime \geq 1$, for any $\mu^\prime \geq 1$, since $c_1(L)=\deg L =-2$. In particular, $\mu^\prime =1$ (corresponding to $\mu =2$) is the lowest positive eigenvalue, so the Reeb vector field is a minimizer with respect to the metric $g$ but not with respect to $g^\prime$. A similar discussion can be done for lens spaces $L(p,1)$, which are circle bundle of degree $-p$ over a $\CC P^1$.
\end{ex}

Notice that the dimension given by the Riemann-Roch argument refers only to a very special subclass of eigenfields (those tangent to the contact distribution) so it may not coincide with the multiplicity of the corresponding eigenvalue.

\begin{ex}[Poincar\'e homology sphere] $M=\Sigma(2,3,5)$ can be seen as a Seifert bundle with Seifert invariants
$\{- 2, 0; (2, 1), (3, 2), (5,4)\}$ and with Chern number equals to
$c_1(L)=\deg(L) + \sum_{j=1}^3 \frac{b_j}{a_j} =  -\frac{1}{30}$.
There are $3$ singular orbits with \emph{orbit invariants} $(p, q)\in\{(2, 1), (3, 1), (5,1)\}$
obtained using the formulae \cite{scot} $a=p, \quad b q \equiv  1 \ (\mathrm{mod} \ p), \quad 0< b < a$.

In general, in the neighbourhood of an exceptional fibre, a Seifert space is the quotient of $\Ss^1 \times D^2$ by the action of $\ZZ_p$ generated by a
homeomorphism which is simply the product of a rotation through $2\pi/p$ on the
$\Ss^1$-factor with a rotation through $2\pi q/p$ on the $D^2$-factor~\cite{scot}.
So if the regular fibre of $M$ is of period $T= 2\pi \tau$, for some $\tau > 0$ that we identify below, then the periods of the exceptional fibres are, respectively $\frac{2\pi}{2}\tau, \quad \frac{2\pi}{3}\tau, \quad \frac{2\pi}{5}\tau$.

As a spherical space form $M=\Ss^3/ \mathbb{I}^*$ where $\mathbb{I}^*$ is the binary icosahedral group of order 120, its volume is $\mathrm{Vol}(M)=2\pi^2/120$.
We also know \cite[$\S$3.4]{boyperf} that $\varphi: M \to \Sigma$ is an orbifold Riemannian submersion over an orbifold surface of Gaussian curvature 4, so that using the Gauss-Bonnet formula combined with Riemann-Hurwitz formula we get
$$4 \mathrm{Area}(\Sigma)= \int_\Sigma K \dif \sigma = 2\pi \chi(\Sigma)=
2\pi \left(2- \sum_i \left(1-\frac{1}{a_i}\right)\right)=2\pi \frac{1}{30}.$$
which yields  $\mathrm{Area}(\Sigma)=\pi/60$. Using the co-area formula, we obtain that the length/period of a typical fibre is $T = \mathrm{Vol}(M)/ \mathrm{Area}(\Sigma) = \pi$. With respect to a $\DD$-homothetic rescale of factor $a=2$ of the standard structure, the regular fibres have length $2\pi$, and Riemann-Roch argument for $\mu^\prime=1$ yields $\dim_\CC H^1(\CC P^1, \Oo(L)) = 1$. Thus ${\mu_1^\prime}^\DD=1$, which corresponds to $\mu_1^\DD=2$ on the standard homology sphere. We conclude that on the Poicar\'e homology sphere with the standard structure we have $\mu_1=2$ as expected, and therefore the Reeb vector is an energy minimizer.
\end{ex}

\section{Proof of Theorem~\ref{T:main1} for irregular Sasakian 3-manifolds}\label{S:irreg}

If the Sasakian structure is irregular, the arguments in Section~\ref{S:qreg} showing the existence of curl eigenfields tangent to the contact distribution do not work. Instead, we shall prove Theorem~\ref{T:main1} in the irregular case using Rukimbira's approximation theorem~\cite{ruk}. It states that any irregular Sasakian structure can be approximated by quasi-regular ones, and the approximation sequence is given by type I deformations of the limiting structure. So let us begin by recalling this type of deformations.

Let $(M, \xi, \eta, \phi, g)$ be a compact Sasakian 3-manifold. \textit{Deformations of type I} are defined by \cite[p.269]{BGalbook}:
\begin{equation}\label{deformI}
\begin{split}
\widetilde{\xi}   =& \ \xi + \rho\,, \\[1.5mm]
\widetilde{\eta} =& \ f \eta , \quad f=\frac{1}{1+\eta(\rho)}\,, \\[1.5mm]
\widetilde{\phi}   =& \ \phi - \frac{\eta(\cdot)}{1+\eta(\rho)} \phi \rho\,, \\[1.5mm]
\widetilde{g}   =& \ \frac12\dif \widetilde{\eta} \circ (\widetilde{\phi} \otimes \mathrm{Id})
+ \widetilde{\eta} \otimes \widetilde{\eta}\,,
\end{split}
\end{equation}
where the perturbation $\rho$ is any vector field giving rise to a contact metric structure. These deformations do not affect the underlying CR structure. It is useful to notice that~\cite[p.271]{BGalbook}
\begin{equation}\label{rho}
\rho= \frac{1-f}{f}\xi -\frac{1}{2}\phi \gr \frac{1}{f}\,.
\end{equation}

The deformed contact metric structure $(\widetilde \xi, \widetilde \eta, \widetilde \phi, \widetilde g)$ is Sasakian if and only if~\cite{OGau} (see also \cite[p.271]{BGalbook})
\begin{equation}
\Hess_{1/f}(\phi X , \phi Y)= \Hess_{1/f}(X , Y)\,,
\end{equation}
for all sections $X$ and $Y$ of the contact distribution $\DD = \Ker \eta$.

A standard example of Sasakian manifold with irregular Reeb field is the weighted sphere (see next section) with positive weights $k$ and $\ell$ such that $k/\ell \notin \mathbb{Q}$.

Rukimbira's theorem~\cite[p.212]{BGalbook} claims that, on a compact manifold, every Sasakian structure is a type I deformation of a quasi-regular one where $\rho$ is an infinitesimal automorphism of the quasi-regular structure as small as one wishes (see the proof of~\cite[Thrm. 7.1.10]{BGalbook}). In particular, $\rho$ leaves invariant the contact form $\eta$, i.e. $L_\rho\eta=0$. We are now ready to prove Theorem~\ref{T:main1} for irregular Sasakian manifolds, but we first state a lemma which will be used in the proof.

\begin{lm}\label{L:zero}
If either $\di \rho = 0$ or $\mathcal{L}_\rho \eta=0$, then $\xi(f)=\rho(f)=0$. In particular, this happens if $\rho$ is an infinitesimal automorphism of the Sasakian structure $(\xi, \eta, \phi, g)$.
\end{lm}
\begin{proof}
By Eq.~\eqref{rho} we have, in general, $\rho(\tfrac{1}{f})=\tfrac{1-f}{f}\xi(\tfrac{1}{f})$, or equivalently $\widetilde{\xi}(\tfrac{1}{f})=\tfrac{1}{f}\xi(\tfrac{1}{f})$. Then it is easy to prove that
$\di \rho = 2 \xi(\tfrac{1}{f})$ and that $\widetilde{\xi}(f)\eta + \mathcal{L}_\rho \eta = 0$ (using the Hessian symmetry above and the contact conditions $\mathcal{L}_\xi \eta =\mathcal{L}_{\widetilde{\xi}} \widetilde{\eta} = 0$).
\end{proof}

\begin{proof}[Proof of Theorem~\ref{T:main1} for irregular Sasakian manifolds]
As mentioned above, we may assume that the irregular Sasakian structure $(\widetilde \xi, \widetilde \eta, \widetilde \phi, \widetilde g)$ on $M$ is a deformation of type I of a quasi-regular Sasakian structure $(\xi, \eta, \phi, g)$, with associated vector $\rho$ that is an infinitesimal automorphism of $(\xi, \eta, \phi, g)$.

From Theorem~\ref{genexist} we infer that on $(M, \xi, \eta, \phi, g)$, there exists a $\mu$-eigenfield $X$ of $\cu$ which is tangent to the contact distribution. We claim that $v = f^{-2}X$ is divergence free on
$(M, \widetilde{g})$. Indeed, under a deformation of type I, the Riemannian volume form changes as follows:
$$
\upsilon_{\widetilde g} = \frac{1}{2} \widetilde \eta \wedge \dif \widetilde \eta = \frac{1}{2} f \eta \wedge (\dif f\wedge \eta + f\dif \eta) = f^2 \upsilon_{g}.
$$
Taking the Lie derivative with respect to $X$ gives $\mathcal{L}_X\upsilon_{\widetilde g} = \mathcal{L}_X(f^2 \upsilon_{g})= X(f^2)\upsilon_{g} + f^2 \mathcal{L}_X\upsilon_{g}$
which in turn implies (via the general relation between divergence of a vector field the the Lie derivative of the volume element)
$$
{\widetilde \di} X \, \upsilon_{\widetilde g} = \big(X(f^2) + f^2 \di X\big) \upsilon_{g},
$$
so that $f^2 {\widetilde \di} X = X(f^2) + f^2 \di X$. Multiplication with $f^{-4}$ of the latter equation gives  $f^{-2} {\widetilde \di} X = -X(f^{-2}) + f^{-2} \di X$, or equivalently $ {\widetilde \di} (f^{-2}X) = f^{-2} \di X$. As $X$ is a $\cu$ eigenvector, it is divergenceless, and it results that ${\widetilde \di} (f^{-2}X)=0$ as claimed.

We now consider the second variation of the energy of $\widetilde{\xi}$ along $v$ (see Prop.~\ref{P:secondvar}):
$$\frac{\dif^2}{\dif t^2}E(\widetilde{\xi}_t)\bigl \vert _{t=0} =\int_M \{\abs{\widetilde{\nabla}_{\widetilde{\xi}} v}_{\widetilde{g}}^2 - \widetilde{R}(v, \widetilde{\xi}, \widetilde{\xi}, v)  \} \upsilon_{\widetilde g}\,,$$
where $\abs{\cdot}_{\widetilde{g}}^2$ denotes the norm with respect to the metric $\tilde g$. Using that $\widetilde{\nabla}_{\widetilde{\xi}} v = [\widetilde{\xi}, v] + \widetilde{\nabla}_{v} \widetilde{\xi}= [\xi, v]+ [\rho, v] - \phi v$, Lemma~\ref{L:zero} and the fact that $X$ is an eigenfield of $\cu$ so that $[\xi, X]=\mu \, \phi X$, we obtain
\begin{equation}
\widetilde{\nabla}_{\widetilde{\xi}} v = (\mu -1) \phi v + [\rho, v]\,.
\end{equation}
On the other hand, as on any Sasakian manifold, $\widetilde{R}(v, \widetilde{\xi}, \widetilde{\xi}, v)= \abs{v}_{\widetilde{g}}^2$.  Therefore the second variation of the energy of $\widetilde\xi$ along $v$ reads:
\begin{equation*}
\begin{split}
\frac{\dif^2}{\dif t^2}E(\widetilde{\xi}_t)\bigl \vert _{t=0} &= \mu (\mu -2)\norm{v}_{L^2}^2 + \norm{[\rho, v]}_{L^2}^2 +2(\mu-1)\big([\rho, v], \phi v \big)_{L^2}\\
& \leq \mu (\mu -2)\norm{v}_{L^2}^2 + \norm{[\rho, v]}_{L^2}^2 +2\abs{\mu-1}\norm{v}_{L^2} \norm{[\rho, v]}_{L^2}\\
& = \left(\norm{[\rho, v]}_{L^2} + (\abs{\mu-1}-1)\norm{v}_{L^2} \right)\left(\norm{[\rho, v]}_{L^2} + (\abs{\mu-1}+1)\norm{v}_{L^2} \right)\,,\\
\end{split}
\end{equation*}
where the $L^2$ scalar product and norm are computed with respect to the metric $\tilde g$.

Now we perform a $\DD$-homothetic deformation of the irregular Sasakian structure
$(\widetilde \xi, \widetilde \eta, \widetilde \phi, \widetilde g)$ and its quasi-regular approximation
$(\xi, \eta, \phi, g)$. In particular, we have to impose that $\rho^\prime = a^{-1}\rho$, so that $\widetilde{\xi}^\prime = a^{-1}\widetilde{\xi}$, $\xi^\prime = a^{-1}\xi$, $\mu^\prime=a^{-1}\mu$ and $\widetilde {g}^\prime=a\widetilde g^{\DD}+a^2\widetilde\eta\otimes\widetilde\eta$. It is easy to check that under this rescaling $v'=v$ and the volume changes as $\upsilon_{\widetilde {g}^\prime}=a^2 \upsilon_{\widetilde g}$, and hence the term $\left(\norm{[\rho, v]}_{L^2} + (\abs{\mu-1}-1)\norm{v}_{L^2} \right)$ becomes:
$$
\left(a^{1/2}\norm{[\rho, v]}_{L^2} + a^{3/2}(\abs{\tfrac{\mu}{a}-1}-1)\norm{v}_{L^2} \right)\,.
$$
Here we have used that $[\rho,v]\in \DD$ because $\rho$ is an infinitesimal automorphism. Taking $a>\mu$, this expression can be written as
$$
\mu a^{1/2}\norm{v}_{L^2}\left(\frac{\norm{[\rho, v]}_{L^2}}{\mu\norm{v}_{L^2}} -1\right)\,.
$$

Finally, let us show that $\rho$ can be chosen small enough such that
$$
\frac{\norm{[\rho, v]}_{L^2}}{\mu\norm{v}_{L^2}} <1\,,
$$
which implies that $\frac{\dif^2}{\dif t^2}E(\widetilde{\xi}^\prime_t)\bigl \vert _{t=0} <0$, so $\widetilde{\xi}^\prime$ is unstable, and hence it is not a minimizer. Indeed, $v$ reads in terms of the irregular Sasakian contact form $\widetilde \eta$ as
$$
v=\frac{1}{(1-\widetilde\eta(\rho))^2}\,X\,.
$$
Rukimbira's theorem allows us to take the perturbation $\rho$ as small as one wishes, so let us take $\rho$ small enough in the $C^1$ norm so that $\norm{v}_{L^2}\geq C\norm{X}_{L^2}$ and $\norm{(1-\widetilde\eta(\rho))^{-2}}_{H^1}\leq 1 + C\norm{\rho}_{H^1}$. Using the inequality $\norm{[\rho,v]}_{L^2}\leq C\norm{\rho}_{H^1}\norm{v}_{H^1}$, where $C$ only depends on the metric $\widetilde g$, and $\norm{X}_{H^1}\leq C\mu\norm{X}_{L^2}$ (actually, in the metric $g$ we have $\norm{X}_{H^1}=(1+\mu)\norm{X}_{L^2}$, so the constant only takes into account that we are computing the $H^1$ norm in the metric $\widetilde g$), we conclude that
$$
\frac{\norm{[\rho, v]}_{L^2}}{\mu\norm{v}_{L^2}} \leq C\norm{\rho}_{H^1}(1+\norm{\rho}_{H^1})<1
$$
provided that $\rho$ is small (the smallness depends on the irregular Sasakian structure $(\widetilde \xi, \widetilde \eta, \widetilde \phi, \widetilde g)$).

Summarizing, we have proved that for large enough $a$ the Reeb field of the $\DD$-homothetic deformation of $(\widetilde \xi, \widetilde \eta, \widetilde \phi, \widetilde g)$ with constant $a$ is unstable. Theorem~\ref{minfirst} and Proposition~\ref{P:imp} then imply that there exists a curl eigenfield tangent to the contact distribution $\widetilde \DD$ and that $\widetilde \xi$ is a minimizer if and only $a\leq \mu_1^{\widetilde \DD}/2$. The theorem then follows.
\end{proof}

\section{Energy minimizing Beltrami fields on weighted 3-spheres}\label{S:weighted}

In this section we analyze in detail an interesting class of Sasakian 3-manifolds, the quasi-regular weighted spheres, where we can prove that the first positive eigenvalue of the curl operator is 2, thus implying that the Reeb field is an energy minimizer, cf. Theorem~\ref{T:main2}. We start recalling its definition and some  geometric properties~\cite{BGalbook}, which takes advantage of the Seifert fibre space structure of $\Ss^3$ \cite{scot}.

Throughout this section, $k$, $\ell$ will be coprime integers such that $0 \leq \ell < k$. We consider the Seifert fibre space structure of $\Ss^3$ realized by the foliation $\Ff_{k, \ell}$ given by the orbits of the $\Ss^1$-action:
\begin{equation}\label{actkl}
\big(\theta, (z_1, z_2)\big) \mapsto \left(e^{\ii \ell \theta} z_1 , e^{\ii k \theta} z_2\right), \quad (z_1, z_2) \in \Ss^3, \theta \in \RR/(2\pi\ZZ)\,,
\end{equation}
with quotient space the weighted projective space $\CC P^1(\ell, k)$, a 2-orbifold with two conical singularities corresponding to the two singular orbits of the Seifert foliation.

We notice that we can see \eqref{actkl} as the flow of the vector field on $\Ss^3$ which in Hopf coordinates $(\cos s \, e^{\ii \phi_1}, \ \sin s \, e^{\ii \phi_2})$, $(s, \phi_1,\phi_2) \in[0, \pi/2] \times [0, 2\pi]^2$ reads
\begin{equation}\label{xikl}
\xi_{k, \ell}=\ell \partial_{\phi_1} + k \partial_{\phi_2}\,.
\end{equation}

The weighted Sasakian structure with weights $w=(k, \ell)$ on $\Ss^3$~\cite{OGau, taka} (see also~\cite[Ex. 7.1.12]{BGalbook}) is a deformation of type I of the standard structure.  The weighted Reeb vector field $\xi_w$ is the field $\xi_{k,\ell}$ above and the associated contact form is given, in terms of the standard one, as $ \eta_{w} =  \varsigma^{-1} \eta$, where
$$ \varsigma(\cos s e^{\ii \phi_1},\sin s e^{\ii \phi_2})= \ell \cos^2 s + k \sin^2 s$$
is a (globally defined) nowhere zero function on $\Ss^3$. The explicit expression of $\phi_w$ can be computed using~\eqref{deformI}. Finally the (adapted) weighted Sasakian metric is given in Hopf coordinates as
\begin{equation}\label{gw}
\begin{split}
g_{w}=& \ \varsigma^{-1}\left(\dif s^2 + \varsigma^{-2} \sin^2 \! s \cos^2 \! s(k\dif \phi_1 -\ell \dif \phi_2)^2\right) \\
&+ \varsigma^{-2}(\cos^2 \! s \, \dif \phi_1 + \sin^2 \! s \, \dif \phi_2)^2\, .
\end{split}
\end{equation}
We note that, with respect to $g_{w}$, $\Ff_{k, \ell}$ is a Riemannian foliation by geodesic circles.

The Sasakian manifold $(\Ss^3,\xi_w, \eta_w, \phi_w, g_{w})$ obtained in this way will be denoted by $\Ss^3_w$ and will be called \textit{the weighted $3$-sphere} of weights $w=(k, \ell)$. Its scalar curvature is given by:
\begin{equation}\label{scalw}
\Scal^{\Ss^3_w}=6-8\left(1+k+\ell-\frac{3k\ell}{\ell \cos^2 s + k \sin^2 s}\right).
\end{equation}
In particular we have the lower bound (recall the assumption $\ell <k$):
\begin{equation}\label{lowscalw}
\Scal^{\Ss^3_w}\geq 8(2\ell - k)-2.
\end{equation}
\begin{re}
The metric considered in~\cite{OGau} is a $\DD$-homothetic deformation with constant $a=\tfrac{1}{2}(k+\ell)$ of the one given in \eqref{gw}.
\end{re}

It is easy to check that the following vector fields define a \emph{global} adapted orthonormal frame on $\Ss^3_w$:
\begin{equation}
\begin{split}
& \xi_w =\ell \partial_{\phi_1}+k\partial_{\phi_2}\,, \\
& X_1 =\sqrt{\varsigma} \big( \cos(\phi_1 + \phi_2)\partial_s +\sin(\phi_1 + \phi_2)(\tan s \partial_{\phi_1} - \cot s \partial_{\phi_2}) \big)\,,\\
& X_2 =\sqrt{\varsigma} \big( \sin(\phi_1 + \phi_2)\partial_s -\cos(\phi_1 + \phi_2)(\tan s \partial_{\phi_1} - \cot s \partial_{\phi_2}) \big)\,.
\end{split}
\end{equation}
The dual, positively oriented coframe is denoted by $\{\eta_w, \omega_1, \omega_2\}$. Notice that when $k=\ell=1$ (i.e. on the round sphere) this is a global orthonormal frame of Killing vector fields which are $\cu$ eigenfields for the eigenvalue $\mu=2$ (of multiplicity 3).

The main result of this section shows that the Reeb field $\xi_w$ is an energy minimizer (this is the second part of Theorem~\ref{T:main2}). Moreover, this theorem has an interesting application in the seemingly unrelated area of field theories admitting topological solitons (see Proposition~\ref{fadd} below).
\begin{te} \label{thweight}
The first positive eigenvalue of the curl operator on the weighted Sasakian $3$-sphere $\Ss^3_w$ is $2$, and is simple except for the round metric case. In particular, the Reeb field $\xi_w$ is an energy minimizer in its adjoint orbit under the action of volume-preserving diffeomorphisms. Moreover, the first positive eigenvalue of curl acting on vector fields tangent to the contact distribution is $\mu^\DD_1=k+\ell$ with multiplicity $2$.
\end{te}

\begin{proof}[Proof of Theorem~\ref{thweight}] We have seen that the weighted Sasakian structure on $\Ss^3$  is quasi-regular and is tailor-made for the Seifert fibre space structure (circle orbi-bundle) over the weighted projective space $\CP(\ell, k)$, with associated complex line bundle denoted by $L$. Since the length of a typical orbit of the Reeb field $\xi_w$ is $2\pi$, the proof of Theorem~\ref{genexist} implies that the $\mu$-eigenfields of the curl operator tangent to the contact distribution correspond to elements of $H^1\left(\CP(\ell, k), \Oo(L^{\mu})\right) \cong H^0\left(\CP(\ell, k), \Oo(L^{-\mu}\otimes K_{\CP(k, \ell)})\right)$, where the isomorphism comes from Serre duality. To ensure that the latter has non-trivial elements it is necessary to have $c_1(L^{-\mu}\otimes K_{\CP(\ell, k)}) \geq 0$, otherwise Kodaira vanishing theorem (for orbifolds) would imply that
$H^0\left(\CP(\ell, k), \Oo(L^{-\mu}\otimes K_{\CP(\ell, k)})\right)=0$. It then follows that:
$$-\mu c_1(L) + c_1(K_{\CP(\ell, k)}) \geq 0\,.$$
Since, cf. ~\cite{amp}, $c_1(K_{\CP(\ell, k)})=-\frac{1}{k}-\frac{1}{\ell}$ and $c_1(L)=-\frac{1}{k\ell}$, we obtain a lower bound for the eigenvalues of curl associated to eigenfields tangent to the contact distribution:
\begin{equation}\label{mu1Dweight}
\mu \geq k+\ell\,.
\end{equation}
In fact, by direct computation we can check that $\varsigma^{3/2} X_1$ and  $\varsigma^{3/2} X_2$ are independent eigenfields of eigenvalue $k+\ell$, so the lower bound is attained: $\mu^\DD_1=k+\ell$. Since $k+\ell \geq 2$, Corollary~\ref{mag2} implies that the first positive eigenvalue of curl on $\Ss_w^3$ is $\mu_1=2$, so the Reeb vector field is an energy minimizer in its adjoint orbit according to Theorem~\ref{minfirst}.

Finally, let us prove that the vector fields $\varsigma^{3/2} X_1$ and $\varsigma^{3/2} X_2$ form a (real) basis of the space of curl eigenfields tangent to the contact distribution corresponding to the eigenvalue $k+\ell$. Since this space is isomorphic to $H^1(\CP(\ell, k), \Oo(L^{k+\ell}))$ (cf. Section \ref{S:qreg}), and $\varsigma^{3/2} X_{1,2}$ are independent vector fields, it is enough to prove that the complex dimension equals 1 (using a Riemann-Roch argument). Recall~\cite{scot} that the oriented Seifert invariants of the fibration of $\Ss^3_w$ over $\CP(\ell, k)$ with fibres given by $\xi_{k,\ell}$ (and associated complex line orbi-bundle $L$) are:
\begin{itemize}
\item $\deg L = b=-1$

\item $a_1 = \ell$ and $k b_1 \equiv 1 (\mathrm{mod} \  \ell)$ with $0<b_1< \ell$.

\item $a_2 = k$ and $\ell b_2 \equiv 1 (\mathrm{mod} \  k)$ with $0< b_2 < k$.
\end{itemize}

Since $k$ and $\ell$ are relatively prime, there exist $x,y\in\mathbb{Z}$, such that $\ell x + ky=1$. We may choose $0<x<k$ and then it follows that $\frac{1}{k}<y+\ell<\ell+\frac{1}{k}$, so $b_1=y+\ell$ and $b_2=x$. We choose to work with reversely oriented Seifert invariants
\[\{\deg L = -1,0; \ (a_1 = \ell, b_1=-y), (a_2 = k, b_2 = k-x)\}\,.\]
in order to have $c_1(L)=-\frac{1}{k\ell}$ as can be easily checked.

In this case the recurrence relation~\cite{blau} $b_i(L^j) \equiv (b_i(L^{j-1})+b_i(L)) \, (\mathrm{mod} \  a_i)$, $0 \leq b_i(L^j) < a_i$, yields:
$$
b_1(L^{k+\ell}) \equiv -y\left( k - \left[\tfrac{k}{\ell}\right]\ell \right) \  (\mathrm{mod} \  \ell); \quad
b_2(L^{k+\ell}) \equiv \ell\left( k - x \right) \  (\mathrm{mod} \  k)\,.
$$
Using the properties of $x$ and $y$, we deduce $b_1(L^{k+\ell}) =\ell -1$ and $b_2(L^{k+\ell}) = k-1$. Injecting these values in \eqref{dimh1} yields $\dim_\CC H^1(\CP(\ell, k), \Oo(L^{k+\ell}))=1$ as claimed (notice that $\deg L^{k+\ell}=-2$).
\end{proof}

\begin{re}
Following an analogous reasoning one can show, for $j <k+\ell$, that $\dim_\CC H^1\left(\CP(\ell, k), \Oo(L^j)\right)=0$, thus obtaining again \eqref{mu1Dweight}.
\end{re}

\subsection{Application to the Faddeev-Skyrme model} The first positive eigenvalue of the $\cu$ operator and its eigenfields are related not only to the $L^2$-energy minimization of vector fields in their $\SD$-orbit but also to the minimizers of the symplectic Dirichlet energy (aka
$\sigma_2$-energy) of mappings with 1-dimensional fibers. This is a manifestation of the duality described in~\cite{slo}. Let us recall that for a map $\varphi$ defined on a 3-manifold $(M^3, g)$ and taking values in a surface $(N^2, J, h)$ with fundamental 2-form $\Omega(X, Y)=h(JX,Y)$, the \textit{symplectic Dirichlet energy} is defined~\cite{sve, svee} as $\mathcal{F}(\varphi)=\frac{1}{2}\norm{\varphi^* \Omega}_{L^2}^2$. Assuming $M$ is a compact, connected, oriented 3-manifold with $H^2(M, \ZZ)=0$, the following topological lower bound can be established \cite{spe, svee}:
\begin{equation}\label{lowboundF}
\mathcal{F}(\varphi) \geq \tfrac{\mathrm{Area}(N)^2}{2}\mu_1 \, Q(\varphi).
\end{equation}
where $\mu_1$ is the first positive eigenvalue of the $\cu$ operator and $Q(\varphi)$ is the Hopf invariant which is constant in each homotopy class of mappings. Finding (local) minimizers of this energy in each homotopy class when  $M=\Ss^3$ (or $\RR^3$ with asymptotic conditions) and $N=\CC P^1 \cong \Ss^2(\frac{1}{2})$ is relevant for the (strongly coupled) Faddeev-Skyrme model in theoretical physics~\cite{fad}; in this case we have $Q(\varphi) \in \ZZ$.

Here we extend this discussion to mappings $\varphi: \Ss^3 \to \CC P^1(\ell, k)$ defined on the 3-sphere and taking values into the 2-orbifold (topological sphere) with two conical singularities of angles $2\pi/k$ and $2\pi/\ell$, called the weighted projective space. We endow the domain with the weighted Sasakian structure of weights $(k, \ell)$, as above, and the codomain with the induced K\"ahler orbifold structure with K\"ahler/area 2-form $\Omega_w$ ~\cite{BGalbook, DGau}. Let $\ov{\Omega}=\frac{1}{\mathrm{Area}(\CC P^1(\ell, k))} \Omega_w$. Since $H^2(\Ss^3, \ZZ)$ is trivial, the closed 2-form $\varphi^* \ov{\Omega}$ must be exact, so $\varphi^* \ov{\Omega} =\dif A$ for some 1-form $A$ that may be further supposed to be coexact. We define the Hopf invariant with the usual formula $Q(\varphi)=\big( A, \ast \dif A \big)_{L^2}$ which produces a homotopy invariant.

We now consider the natural projection to the quotient space $\pi_w: (\Ss^3, g_w) \to (\CC P^1(\ell, k), h_w)$, $\pi_w(z_1, z_2) = [z_1, z_2]_w$ which is, by construction \cite{BGalbook}, an \textit{orbifold Riemannian submersion} with geodesic fibres tangent to $\xi_w$, fitting into the general Boothby-Wang construction (see Appendix). Let us evaluate the terms involved in \eqref{lowboundF} for this particular mapping\footnote{By composing $\pi_w$ with the natural holomorphic map $f:\CC P^1(\ell, k) \to \CC P^1$, $f([z_1, z_2]_w)=[z_1^k, z_2^\ell]$ of degree 1, we obtain a smooth mapping $\varphi_{k, \ell}:= f \circ \pi_w: \Ss^3 \to \CC P^1$ with smooth codomain, and having the same Hopf invariant. Nevertheless we have to import on $\CC P^1$ the induced orbifold structure (see \cite{abr} for details) and the weighted projective metric in order to maintain the minimization property that we prove here.}.
First, the energy can be directly computed as $\mathcal{F}(\pi_w)=
\tfrac{1}{2}\int_{\Ss^3}\abs{\pi_w^* \Omega_{w}}^2\upsilon_{g_w}=
\tfrac{1}{2}\Vol(\Ss^3_w)=\frac{\pi^2}{k\ell}$, where we used $\pi_w^* \Omega_w = -\frac{1}{2}\dif \eta_w$. The latter property together with the formula
$\mathrm{Area}(\CC P^1(\ell, k))=\frac{\pi}{k\ell}$, c.f.~\cite{DGau} (or deduced from the co-area formula), allow us to compute the Hopf invariant and to obtain $Q(\pi_w)=k\ell$. Finally, since we have proved that $\mu_1=2$, we can now see that our map $\pi_w$ realizes the equality in Eq.~\eqref{lowboundF}. Accordingly, we have established the following:
\begin{pr}\label{fadd}
For any positive coprime integers $k, \ell$, the natural projection $\pi_w:\Ss_w^3 \to \CC P^1(\ell, k)$ has the lowest symplectic Dirichlet energy in its homotopy class.
\end{pr}

\section{Appendix: orientation and other conventions}

In this paper we adopt the notations and conventions that we found to be the most common in contact geometry, c.f.~\cite{BGalbook} (notice though that there the differential of a 1-form includes a 1/2 factor, while in our work $\dif\eta(X,Y)=X(\eta(Y))-Y(\eta(X))-\eta([X,Y])$).

If our compact Sasakian 3-manifold $(M, \xi, \eta, \phi, g)$ is quasi-regular, then~\cite[Theorem 7.1.3]{BGalbook} $M$ is a Seifert fibration (principal $\Ss^1$ orbibundle) over a Hodge 2-orbifold $\Sigma_g$ and the contact metric data and the fibration data will be related as follows:
\begin{itemize}
\item[(i)] $\Sigma_g$ can be endowed with a K\"ahler structure $(J, h, \Omega)$,
$\Omega(\cdot, \cdot) := h(J\cdot, \cdot)$, such that the canonical projection $\pi : M \to \Sigma_g$ is a Riemannian submersion, i.e. $g = \pi^* h + \eta \otimes \eta$, and $(\phi, J)$-holomorphic, i.e. $\dif \pi \circ \phi = J \circ \dif \pi$.
In particular, from these conditions and from $\tfrac{1}{2}\dif \eta(\cdot, \cdot) = g(\cdot, \phi \cdot)$ (see Section~\ref{S:intro}), it follows that $\tfrac{1}{2}\dif \eta = - \pi^*\Omega$.

\item[(ii)]  The action of $\Ss^1$ on $M$ on the right is considered positive: $(t, x) \mapsto x \cdot e^{\ii t}$, $t \in \RR/2\pi \ZZ$, $x\in M$ and $-\xi$ is the infinitesimal generator of the $\Ss^1$-action. Therefore $\theta = - \ii \eta$ yields a connection $1$-form $\theta:TM \to \ii \RR$ in the principal $\Ss^1$ orbibundle.
\end{itemize}

We recognize in the above definition (a version of) the generalized \textit{Boothby-Wang  construction}, aka Kaluza-Klein construction\footnote{The unusual minus signs in (i) and (ii) are due to the fact that we wanted to keep the ``asymmetrical'' conventions $\tfrac{1}{2}\dif \eta(\cdot, \cdot) = g(\cdot, \phi \cdot)$, and $\Omega(\cdot, \cdot) = h(J\cdot, \cdot)$ which are commonly used in Sasakian and K\"ahler geometry context, respectively.}. The orientation that we have chosen on $M$ (see Section~\ref{S:intro}) is defined by requiring the Riemannian volume element to be $\frac{1}{2}\eta \wedge \dif \eta$ so that a positive oriented (orthonormal) frame on $M$ is $\{\xi, X_1, -\phi  X_1\}$, while on $\Sigma_g$ a positive frame will be $\{Y_1, JY_1\}$, if $Y_1 = \dif \pi(X_1)$. In other words: orientation on $M$ = reverse fibres orientation $\wedge$ reverse orientation from the base.

\begin{re}[Sign of the Chern number]  With the above conventions, for the complex line bundle $L$ associated to the principal $\Ss^1$ orbibundle $M$ over $\Sigma_g$, we always have $c_1(L) < 0$. (Compare to the similar definition in \cite{bel1} where different conventions lead to $c_1(L)$ always positive.)
Indeed, since the curvature of the fibration is $\dif \theta= -\ii \dif \eta = 2\ii \pi^*\Omega$, we have
$c_1(L)=\frac{\ii}{2\pi}\int_{\sigma_g} 2\ii \Omega = -\frac{1}{\pi}\mathrm{area}(\Sigma_g) <0$. Notice also that, by the co-area formula, as the regular fibres have period (and length) $2\pi$ we also have: $\int_M \eta \wedge \dif \eta =-4\pi^2 c_1(L)$. In particular, $M=\Ss^3$ is the total space of the Hopf fibration over $\Sigma_0=\CC P^1 \cong \Ss^2(\frac{1}{2})$ and the Boothby-Wang construction above (with the standard metric and orientation given by $J$ on the base) corresponds to the standard Sasakian structure \cite[p.210]{BGalbook} with contact form $\eta = \tfrac{\ii}{2}(\ov{z}_1 dz_1 - z_1 d \ov{z}_1 + \ov{z}_2 dz_2 - z_2 d \ov{z}_2)$ and, as $\mathrm{area}(\Sigma_0)=\pi$, we have $c_1(L)= -1$.
\end{re}

We end this section by proving the correspondence between (1,0)-forms $\omega$ (associated to curl eigenfields in $\DD = \ker \eta$) and holomorphic forms on $\Sigma_g$ with values in $L^{-\mu}$. Indeed, a form $\omega \in \Lambda_\DD^{(1,0)}M$ with the property $\mathcal{L}_\xi\omega=-\ii\mu\omega$ must satisfy (via integration along the flow $\varphi_{t}$ of $\xi$)
$$
\varphi_{t}^{*}\omega = \exp(-\ii \mu t)\omega.
$$
This can be read as $R_{a^{-1}}^* \omega = \rho(a)\omega$, with $a=e^{\ii t}$ and $\rho(a)z=a^{-\mu}z$, which, together with $\imath_\xi \omega = 0$, says that $\omega$ is a tensorial form of degree 1 and of type $(\rho, \CC)$ (see e.g. \cite[p.20]{BGalbook}). Therefore, by the general theory of principal and associated bundles, such a form $\omega$ can be seen as a section of  $\Lambda^{(1,0)}\Sigma_g \otimes L^{-\mu}$. Notice that the (1,0) type of $\omega$ was maintained due the holomorphicity of the bundle map $\pi$; for the same reason if $\omega$ is  holomorphic ($\ov \partial \omega = 0$), then the same is true for $\omega$ as section in $K_{\Sigma_g}\otimes L^{-\mu}$, joining the statement in \cite[p. 619]{biq} (see also \cite[Prop. 12.7]{tana}).

\begin{re}[Negative eigenvalues]
Considering in the above construction either

$\bullet$ the negative action of $\Ss^1$ on $M$, or

$\bullet$ $\Omega(\cdot, \cdot) := h(\cdot, J \cdot)=$ K\"ahler/positive area 2-form on the base (i.e. consider the opposite complex structure $J$) and $\xi =$ the infinitesimal generator of the $\Ss^1$-action,
while keeping the other conventions unchanged, will have the following effect: $(1,0)$-forms $\omega$ (associated to curl eigenfields in $\DD = \ker \eta$) will correspond this time to holomorphic forms on $\Sigma_g$ with values in $L^{\mu}$, i.e. elements of $H^1(\Sigma_g, \Oo(L^{-\mu}))$,and the sign of the Chern number will remain the same ($c_1(L) < 0$). So the Riemann-Roch argument can be applied also for negative eigenvalues of the $\cu$ operator.
\end{re}

\end{document}